\documentclass[reqno,draft]{amsart}
\usepackage{amsmath,amssymb,bbm}
\usepackage{graphicx}
\usepackage[usenames]{color}

\newtheorem{theorem}{Theorem}[section]
\newtheorem{lemma}[theorem]{Lemma}
\newtheorem{corollary}[theorem]{Corollary}

\theoremstyle{remark}
\newtheorem{remark}[theorem]{Remark} 
 
\newtheorem{example}[theorem]{Example}

\definecolor{myred}{rgb}{1,0.1,0.2}

\newcommand {\mybf}{}

\numberwithin{equation}{section}

\newcommand{\kG}{{\mathcal G}}
\newcommand{\kX}{{\mathcal X}}
\newcommand{\kE}{{\mathcal E}}
\newcommand{\kY}{{\mathcal Y}}
\newcommand{\kB}{{\mathcal B}}
\newcommand{\kC}{{\mathcal C}}
\newcommand{\kV}{{\mathcal V}}
\newcommand{\kW}{{\mathcal W}}
\newcommand{\kD}{{\mathcal D}}
\newcommand{\kL}{{\mathcal L}}
\newcommand{\kH}{{\mathcal H}}
\newcommand{\kU}{{\mathcal U}}

\newcommand{\kI}{{\mathcal I}}

\newcommand{\kZ}{{\mathcal Z}}

\newcommand{\kCb}{\kC_{\operatorname{b}}}

\newcommand  {\D}{{\mathrm D}}
\renewcommand{\d}{{\mathrm d}}
\renewcommand{\i}{{\mathrm i}}
\newcommand  {\e}{{\mathrm e}}
\newcommand {\rn}{{\mathrm n}}
\newcommand {\rb}{{\mathrm b}}

\renewcommand{\a}{{\mathsf a}}
\newcommand {\mS}{{\mathsf S}}
\newcommand {\mP}{{\mathsf P}}
\newcommand {\mQ}{{\mathsf Q}}
\renewcommand{\b}{{\mathsf b}}

\newcommand  {\1}{\mathbbm{1}}
\newcommand  {\R}{{\mathbb R}}
\newcommand  {\N}{{\mathbb N}}
\newcommand  {\C}{{\mathbb C}}
   
\newcommand  {\Z}{{\mathbb Z}}
\renewcommand{\P}{{\mathbb P}}
\newcommand  {\Q}{{\mathbb Q}}

\newcommand {\id}{\operatorname{id}}
\newcommand {\spec}{\operatorname{spec}}

\renewcommand {\Re}{\operatorname{Re}}
\renewcommand {\Im}{\operatorname{Im}}
\newcommand{\dist}{\operatorname{dist}}
\newcommand{\interior}{\operatorname{int}}
\newcommand{\cl}{\operatorname{cl}}
\newcommand{\nb}{{\operatorname{nb}}}

\newcommand{\norm}[3][]{%
  \mathchoice{\lVert #2 \rVert_{#3}^{#1\vphantom\int}}%
             {\lVert #2 \rVert_{#3}^{#1}}{}{}}
\newcommand{\tnorm}[2]{%
  \lVert #1 \rVert_{#2}}
\def\vvvert{\hbox{\ensuremath{|\hspace{-0.16em}|\hspace{-0.16em}|}}}
\newcommand{\Norm}[2]{%
  \mathchoice{\vvvert #1 \vvvert\vphantom\vert_{#2}^{\vphantom\int}}%
             {\vvvert #1 \vvvert\vphantom\vert_{#2}}{}{}}

\begin{document}
\title{\emph{A}-stable Runge--Kutta methods for semilinear evolution
equations} 

\author[M. Oliver]{Marcel Oliver}
\address[M. Oliver]%
{School of Engineering and Science \\
 Jacobs University \\
 28759 Bremen \\
 Germany}
\email{oliver@member.ams.org}

\author[C. Wulff]{Claudia Wulff}
\address[C. Wulff]%
{Department of Mathematics\\
 University of Surrey \\
 Guildford GU2 7XH \\
 UK}
\email{c.wulff@surrey.ac.uk}

\keywords{Semilinear evolution equation, smoothness of semiflow,
$A$-stable Runge Kutta semidiscretization in time}

\date{\today}

\begin{abstract}  
We consider semilinear evolution equations for which the linear part
generates a strongly continuous semigroup and the nonlinear part is
sufficiently smooth on a scale of Hilbert spaces.  In this setting, we
prove the existence of solutions which are temporally smooth in the
norm of the lowest rung of the scale for an open set of initial data
on the highest rung of the scale. Under the same assumptions, we prove
that a class of implicit, $A$-stable Runge--Kutta semidiscretizations
in time of such equations are smooth as maps from open subsets of the
highest rung into the lowest rung of the scale.  Under the additional
assumption that the linear part of the evolution equation is normal or
sectorial, we prove full order convergence of the semidiscretization
in time for initial data on open sets. Our results apply, in
particular, to the semilinear wave equation and to the nonlinear
Schr\"odinger equation.
\end{abstract}

\maketitle

\tableofcontents

\section{Introduction}
\label{s.intro}
We study numerical schemes for evolution equations on Hilbert spaces
by first looking at the properties of a semidiscretization in time
only; discretization in space is then treated as a perturbation within
the Hilbert space setting.  This approach was introduced by Rothe
\cite{Rothe}.  When successful, results so obtained are naturally
uniform in the spatial discretization parameter.  In contrast, when
first discretizing in space, the resulting finite dimensional system
of ordinary differential equations (ODEs) can be treated with
ODE-based techniques which avoids the difficulty arising from the
analysis of equations on infinite-dimensional spaces, but where
uniformity in the spatial mesh size is not immediate.

In this paper, we consider semilinear evolution equations of the form
\[
  \partial_t U = AU + B(U)
\]
on a \mybf{Hilbert} space $\kY$.  The linear operator $A$ is assumed
to generate a strongly continuous, not \mybf{necessarily} analytic
semigroup and $B$ is a bounded nonlinear operator on $\kY$.  The
examples we have in mind are semilinear Hamiltonian evolution
equations such as the semilinear wave equation or the nonlinear
Schr\"odinger equation \mybf{with periodic, homogeneous Dirichlet, or
homogeneous Neumann boundary conditions, or on the line.}  However,
the results in this paper do not depend on a Hamiltonian structure.

We analyze the differentiability properties in initial value and time
step of \mybf{the semiflow of the evolution equation and of}
a large class of $A$-stable Runge--Kutta methods, including the
Gauss--Legendre methods, when applied to the evolution equation.  To
be able to differentiate the semiflow and the numerical method we
formulate conditions that guarantee uniformity of the time-interval of
existence (for the semiflow) and the maximum step size (for the
numerical methods) over bounded sets of parameters.  We present two
versions of such uniformity results: Whenever existence can be
achieved, uniformity holds on sufficiently small balls of initial
data; we will label results of this type by ``local version.''
Assuming more regularity for the initial data, we also obtain results
which are uniform on bounded open sets so long as $B$ is well-defined
and bounded.  We will label results of this type by ``uniform
version.''

Note that differentiation in time results in multiplication with the
unbounded operator $A$ and is only well-defined when considered as a
map from a subset of $D(A)$ to $\kY$; this is easily seen by taking $B
\equiv 0$ and differentiating the exact semiflow $\e^{tA}U^0$.  To be
able to differentiate repeatedly in time we assume that $B$ is
$\kC^{N-k}$ as map from some open set $\kD_k \subset \kY_k \equiv
D(A^k)$ to $\kY_k$ for $k=0, \ldots, K$ and $N>K$.  Whether or not
this condition is satisfied depends on the given evolution equation,
and in particular on the boundary conditions; it is satisfied for the
equations mentioned above in the case of periodic boundary conditions
and smooth nonlinearities.  \mybf{We also give examples of PDEs with
Neumann boundary conditions, Dirichlet boundary conditions, and on the
line where this condition is true.}  We then prove that the semiflow
of the evolution equation and the numerical method are of class
$\kC^K$ jointly in time (resp.\ step size) and initial data when
considered as a map from $\kD_K$ to $\kY$.  Both results require
carefully tracking the domains of definition of $B$.  Moreover, under
the additional assumption that $A$ is normal \mybf{(or, more
generally, normal up to a perturbation which is a bounded linear
operator on each of the $\kY_k$) or that $A$ is sectorial}, we show
convergence of the semidiscretization in time \mybf{at its full order
$p$ provided $K=p$ and for initial data $U^0 \in \kD_{K+1}$. 

The exact solution $U(t)$ of the semilinear evolution equation is
obtained as a fixed point of a contraction map.  Similarly, the
Runge--Kutta methods we consider are implicit as they are functions of
the Runge--Kutta stage vectors, which in turn are obtained as fixed
points of contraction maps.  As for the exact solution,
differentiation in the step size of the Runge--Kutta method results in
multiplication by the unbounded operator $A$.  Hence, these
derivatives are also only well-defined on the scale of \mybf{Hilbert} spaces
$\kY_k$.  An additional difficulty arises from the fact that the
semiflow, the numerical method, the contraction maps for semiflow and
stage vectors, and their derivatives with respect to the initial data
are only strongly continuous in the time-like parameter, but not
continuous in the operator norm.  Hence, these maps do not fit into
the usual setting of contraction mapping theorems with parameters.  We
therefore address these issues by providing an abstract theory for the
differentiability properties of fixed points of contraction mappings
on a scale of Banach spaces.  This theory provides a unified framework
for the time-continuous and time-semidiscrete case.
 
Let us mention some related results in the literature.  Le Roux
\cite{LeRoux} studies convergence results for strongly $A$-stable
approximations $\mS(hA)$ of holomorphic semigroups $\e^{hA}$ on
Hilbert spaces, an example of which are strongly $A$-stable
Runge--Kutta methods applied to linear parabolic systems.  Palencia
\cite{Palencia} and Crouzeix \emph{et al.}\ \cite{Crouzeix93} study
stability of A-acceptable rational approximations $\mS(hA)$ of
holomorphic semigroups $\e^{hA}$ on Banach spaces; they show that when
$\Re (\spec A ) \leq \omega$ for some $\omega>0$, then $\lVert
\mS^n(hA) \rVert \leq \Theta_\mS \, \e^{\omega_\mS nh}$ for some
$\omega_\mS>0$, $\Theta_\mS>0$, and all $n \in \N$.  Lubich and
Ostermann \cite{LubOst93} prove convergence results for Runge--Kutta
methods applied to semilinear parabolic equations on Banach spaces,
cf.\ \cite{CrouzeixSemilinPar}.  Variable step size schemes applied to
fully nonlinear parabolic problems have been studied in
\cite{OstQuasilinVarStepRK}.  Gonz\'alez and Palencia
\cite{RKQuasilinPar} study stability of $A$-stable Runge--Kutta
methods in the initial value, as we do, but they study quasilinear
parabolic problems and do not consider the differentiability
properties of the solution.  Akrivis and Crouzeix
\cite{CrouzeixMultistep} discuss multistep semidiscretizations in time
for parabolic problems; see references therein for further related
work.

In \cite{HershKato}, quoted above, Hersh and Kato prove convergence of
$A$-acceptable rational approximations $\mS(hA)$ of non-analytic
$C_0$-semigroups $\e^{hA}$ for smooth initial data.  Brenner and
Thom\'ee \cite{BrennerThomee} show that $A$-acceptable rational
approximations $\mS(hA)$ of non-analytic $C_0$-semigroups $\e^{hA}$
\mybf{with $\Re (\spec A) \leq 0$} in general grow like $\lVert
\mS^n(hA) \rVert = O(n^{1/2})$ and study fractional order convergence
for non-smooth initial data of linear evolution equations, see also
\cite{Kovacs07}; for extensions to variable step size, see
\cite{VarStepsizeRK}.  Brenner \emph{et al.}\ \cite{BCT82} study
convergence of rational approximations of inhomogeneous linear
differential equations on Banach spaces, assuming stability of the
approximation.  Colin \emph{et al.}\ \cite{Time-semi-NLS,Time-semi-Zakh}
study modified Crank--Nicolson semidiscretizations in time of
nonlinear Schr\"odinger equations and Zakharov wave equations.  They
prove convergence as $h\to 0$, but do not analyze the order of
convergence.

In this paper we have a related, but different objective.  Similarly
as in \cite{BrennerThomee,HershKato,Kovacs07} we consider
semidiscretizations in time of evolution equations which are not
parabolic, i.e., equations whose linear part does not generate an
analytic semigroup.  But whereas the main issue in
\cite{BrennerThomee,HershKato,Kovacs07} is that the spectral theorem
is not available for the linear operator $A$ so that stability
estimates of the form \mybf{$\lVert \mS^n(hA) \rVert \leq \Theta_\mS \,
\e^{\omega_\mS nh}$}, as required for the standard convergence analysis,
are not available, we assume here, like \cite{BCT82}, that this
estimate holds true, e.g.\ due to normality of $A$ \mybf{on the Hilbert
space $\kY$.}  Our focus is rather on semilinear problems as were
considered by Lubich and Ostermann \cite{LubOst93} in the parabolic
case; the class of Runge--Kutta schemes considered here is the same as
in their work.  \mybf{However, while \cite{LubOst93,CrouzeixSemilinPar}
assume the existence of a temporally smooth solution $U(t)$ of the
semilinear evolution equation (or a perturbation of it)
to obtain higher order convergence, we
provide a detailed analysis under which conditions this
assumption holds true.}

Our conditions on $B$ yield, in particular, $\kC^K$ smoothness of the
semiflow jointly in time and in the initial data for initial values in
an open set of a Hilbert space $\kY_K$.  If the conditions on $B$ are
not satisfied, the set of initial values of temporally smooth
solutions is generally a complicated set which is characterized by
nonlinear conditions; hence such initial data are in general difficult
to prepare numerically.  We illustrate this for the semilinear wave
equation with generic nonlinearity and Dirichlet boundary conditions;
see Section~\ref{sss.Dirichlet-NWE}.  Under the same conditions on
$B$, differentiability of the numerical method in the step size $h$
and in the initial data holds on open sets.  This allows us to prove
convergence of the numerical method without additional stage order
conditions as have been assumed in \cite{LubOst93}.  Moreover, we
obtain full order convergence, whereas e.g.\ the convergence results
for semilinear PDEs of \cite{LubOst93} provide an order of convergence
that is determined by the stage order and that, in general, is smaller
than the order of the method.

Lubich and Ostermann \cite{LubOst93}, in the parabolic setting, also
obtain convergence results when only a perturbation of the continuous
solution $U(t)$ is temporally smooth, and their estimate of the
trajectory error then also depends on this perturbation error.  In
practice, such a perturbation would typically be a space
discretization; if the continuous solution lacks temporal smoothness,
the assumption of a temporally smooth solution of a perturbation
tending to zero with the step size $h$ typically imposes mesh
conditions that exclude order $p$ convergence of the
semidiscretization in time.

In contrast to \cite{BrennerThomee,LubOst93}, our interest in this
paper is not on fractional order} convergence for non-smooth initial
data.  Rather, since we are interested in obtaining higher order
differentiability of the numerical method in the time step, we
restrict attention to regular initial data $U^0 \in \kY_{K+1}$; in
particular, we assume enough regularity to have full order of
convergence, i.e., $K \geq p$ where $p$ is the order of the numerical
method.  Our convergence result extends the corresponding classical
result for linear evolution equations of Hersh and Kato
\cite{HershKato} to nonlinear systems.

There has been a lot of recent activity in the application of split
step time-semidiscretizations of nonlinear Schr\"odinger and wave
equations: Besse \emph{et al.}\ \cite{Besse2002} and Lubich
\cite{Lubich2008} study convergence of split step
time-semidiscretizations for nonlinear Schr\"odinger equations; also
see \cite{Ostermann09} for a general framework in the linear case and
more references, and \cite{Gaukler,Faou-splitPDE} for long-time
preservation of actions of nonlinear Schr\"odinger equations under
split step time-semidiscretizations.  While splitting methods are very
effective for simulating evolution equations for which the linear
evolution $\e^{tA}$ can easily be computed explicitly, Runge--Kutta
methods are still a good choice when an eigendecomposition of $A$ is
not available, as for example for the semilinear wave equation in an
inhomogeneous medium; see Section~\ref{sss.swe-nonhom}.  Moreover, the
simplest example of a Gauss--Legendre Runge-Kutta method, the implicit
mid point rule, appears to have some advantage over split step
time-semidiscretizations for the computation of wave trains for
nonlinear Schr\"odinger equations \cite{Herbst,Herbst2} because the
latter introduce an artificial instability while the former reproduces
recurrences well.

In this paper, we shall hence restrict our
attention to Runge--Kutta methods which have a long history as robust
and effective time integrators for both ODEs and PDEs; see, e.g.,
\cite{Smith86,CFD}.  Gauss--Legendre Runge--Kutta methods, in
particular, have attracted attention as they are symplectic and yield
multisymplectic space-time schemes for PDEs \cite{BridgesReich06}.

\mybf{While we restrict attention to evolution equations on Hilbert
spaces $\kY$, our results on differentiability of the semiflow and of
the numerical method also hold true when $\kY$ is a Banach
space.  However, the stability condition $\lVert \mS^n(hA) \rVert \leq
\Theta_\mS \, \e^{n \omega_\mS h}$, which we need for our convergence
result, is quite restrictive in the Banach space setting, as discussed
above.}

The paper is organized as follows.  In Section~\ref{s.pde}, we
introduce the class of semilinear evolution equations considered,
and study the differentiability
properties of the semiflow of these evolution equations.  We also 
\mybf{present a general result on  
the differentiability of superposition operators. We then show}  
how the semilinear wave equation and the nonlinear Schr\"odinger
equation  with different types of boundary conditions
fit into this framework.  In
Section~\ref{s.rk}, we derive corresponding statements on the
well-posedness, differentiability properties, and convergence of
$A$-stable Runge--Kutta methods when applied to such evolution
equations.  In Appendix \ref{s.appendix}, we present a number of technical results,
most notably a contraction mapping theorem on a scale of Banach
spaces, which are needed in the main body of the paper.


\section{Semilinear evolution equations}
\label{s.pde}
 
In this section, we set up the framework for a class of semilinear
evolution equations whose time discretization we analyze subsequently.
After introducing some notation (Section~\ref{ss.notation}) and
setting up the general functional framework in
Sections~\ref{ss.genSett}, we provide a setting in which the semiflow
is differentiable with respect to the initial data as well as time
(Section~\ref{ss.diffSemiflow}).  In many examples, the nonlinearities
are superposition operators of nonlinear functions; we collect their
fundamental properties in Section~\ref{ss.superpos}.  These results
enable us to fit our two main examples, the semilinear wave equation
(Section~\ref{ss.swe}) and the nonlinear Schr\"odinger equation
(Section~\ref{ss.nse}), into the abstract framework.


\subsection{Some notation}
\label{ss.notation}

Let $\kY$ be a Banach space.  We write
\[
  \kB^\kY_R (U^0) =  \{ U \in \kY \colon \norm{U-U^0}{\kY} \leq R \} 
\]
to denote the closed ball of radius $R$ around $U^0 \in \kY$.  (If no
confusion about the space is possible, we may drop the superscript, or
write $\kB_R(U_0) \subset \kY$ instead of $\kB_R^\kY(U_0)$.)  Let $\kD
\subset \kY$ be open.  We define
\begin{equation}
  \label{e.kD-delta}
  \kD^{-\delta}  
  = \{ U \in \kD  \colon 
       \dist_{\kY } (U, \partial \kD ) > \delta \} \,,
\end{equation}
\mybf{where $\dist_{\kY}(U, \kD) = \inf_{W \in \kD} \norm{U-W}{\kY}$ denotes
the distance between a point $U \in \kY$ and the set $\kD \subset \kY$
measured in the $\kY$-norm.}

For Banach spaces $\kX$ and $\kY$, and $j \in \N_0$, we write
$\kE^j(\kY,\kX)$ to denote the vector space of $j$-multilinear bounded
mappings from $\kY$ to $\kX$; we set $\kE^j(\kX) \equiv
\kE^j(\kX,\kX)$.

For Banach spaces $\kX$, $\kY$, and $\kZ$, and open subsets $\kU
\subset \kX$, $\kV \subset \kY$, and $\kW \subset \kZ$, we write
\[
  F \in \kC^{(\underline{m},n)} (\kU \times \kV; \kW)
\]
to denote a continuous function $F \colon \kU \times \kV \to \kW$
whose partial Fr\'echet derivatives $\D_X^i \D_Y^j F(X,Y)$ exist and
are such that the maps
\begin{equation}
  \label{e.derivativemaps}
  (X, Y, X_1,\ldots, X_{i}) 
  \mapsto \D_X^i \D_Y^j  F(X,Y)(X_1, \ldots, X_{i}) 
\end{equation}
are continuous from $\kU \times \kV \times \kX^{i}$ into $\kE^j(\kY,
\kZ)$ for $i = 0, \dots, m$ and $j = 0,\dots, n$.  In particular, all
directional derivatives are continuous.   We write
\[
  F \in \kCb^{(\underline{m},n)}(\kU \times \kV; \kW)
\]
if, in addition, the partial Fr\'echet derivatives are bounded and the
maps \eqref{e.derivativemaps} extend continuously to the boundary.
(The latter is important as we will apply the contraction mapping
theorem to maps in such classes.)  As usual, we write
\[
  F \in \kC^{(m,n)} (\kU \times \kV; \kW)
\]
to denote that the partial Fr\'echet derivatives up to order $(m,n)$
exist and are continuous in the norm topology; we write $\kCb^{(m,n)}$
if these derivatives are, in addition, bounded and extend continuously
to the boundary.  If any of the sets is not open, we define
\[
  \kC^{(m,n)}(\kU \times \kV; \kW) 
  \equiv \kC^{(m,n)}(\interior(\kU) \times \interior(\kV); 
    \interior(\kW)) \,,
\]
where $\interior(\kU)$ denotes the interior of $\kU$, with analogous
notation for the $\kCb$-spaces.  
\mybf{ The spaces $\kC^m(\kU;\kW)$ and $\kCb^m(\kU;\kW)$ are defined likewise.}

Note that $\kC^{(\underline{m},n)}(\kU \times \kV; \kW) =
\kC^{(m,n)}(\kU \times \kV; \kW)$ only if $\kX$ is finite-dimensional.
In general,
\begin{subequations}
\begin{equation}
  \label{e.cont-a}
  \kCb^{({m},n)} (\kU \times \kV; \kW)
  \supset \kCb^{(\underline{m+1},n)} (\kU \times \kV; \kW)
  \cap \kCb^{(\underline{m},n+1)} (\kU \times \kV; \kW)
\end{equation}
because any differentiable function is continuous.  Moreover,
\begin{equation}
  \label{e.cont-b}
  \kCb^{(\underline{0},k)} (\kU \times \kV; \kW) = 
  \kCb^{({0},k)} (\kU \times \kV; \kW) \,.
\end{equation}
\end{subequations}
In the above, $\kV$ will typically be some interval of time.  

\subsection{General setting}
\label{ss.genSett}

We consider semilinear evolution equations on a \mybf{Hilbert} space $\kY$,
\begin{equation}
  \label{e.pde}
  \partial_t U  = F(U)  = AU + B(U) \,,
\end{equation}
where $U \colon [0,T] \to \kY$.  Equation \eqref{e.pde} formally
looks like an ODE, but will be thought of as being posed on an
infinite-dimensional function space $\kY$. 

Our main examples are the following.

\begin{example}[Semilinear wave equation] \label{ex:SWE}
For the semilinear wave equation 
\begin{equation}
  \label{e.swe}
  \partial_{tt}u = \partial_{xx} u - f(u) \,,
\end{equation} 
we write $v = \partial_t u$ and $U=(u,v)^T$ which, for $t$ fixed,
shall be an element of a Hilbert space $\kY$ to be specified later, so
that
\begin{equation}
  \label{e.AB}
  A =
  \begin{pmatrix}
  0 & \id \\ \partial^2_x & 0
  \end{pmatrix} 
  \quad \text{and} \quad
  B(U) = 
  \begin{pmatrix}
    0 \\ -f(u)
  \end{pmatrix} \,.
\end{equation}
\end{example}

\begin{example}[Nonlinear Schr\"odinger equation] \label{ex:NSE}
For the nonlinear Schr\"odinger equation
\begin{equation}
  \label{e.nse}
  \i \, \partial_t u 
  = - \partial_{xx} u + \partial_{\overline{u}} V(u,\overline{u}) \,,
\end{equation}
we set $U \equiv u$, so that
\begin{equation}
  \label{e.nlsdefs}
  A = \i \, \partial^2_x 
  \quad \text{and} \quad
  B(U) = -\i \, \partial_{\overline{u}} V(u,\overline{u}) \,. 
\end{equation}
\end{example}

In the following, we introduce the framework in which we obtain
 smooth solutions of \eqref{e.pde}.  Later, in
Sections~\ref{ss.swe} and~\ref{ss.nse}, we show how the semilinear
wave equation and the nonlinear Schr\"odinger equation as formally
introduced above fit into this framework.  It is well known that the
following conditions imply the existence of a semiflow of
\eqref{e.pde}.

\begin{itemize}
\item[(A0)] $A$ is a closed, densely defined linear operator on $\kY$
and generates a $C_0$-semigroup on $\kY$.
\item[(B0)] $B \colon \kD \to \kY$ is Lipschitz on some open set $\kD
\subset \kY$.
\end{itemize}
For the definition of strongly continuous semigroups
($C_0$-semigroups) and detailed proofs, see, e.g., \cite{Pazy}.  For
our purposes, the main points can be summarized as follows.

Condition (A0) implies, in particular, that there exist constants
$\omega$ and $\Theta$ such that for every $t \geq 0$
\begin{equation}
  \label{e.etAEstimate}
  \lVert \e^{tA} \rVert \leq \Theta \, \e^{\omega t} 
\end{equation}
with $\Re (\spec A)\leq \omega$.  Moreover, for every $\lambda \in \C$
with $\Re \lambda > \omega$,
\begin{equation}
  \label{e.ResolventEst}
  \lVert (\lambda -A)^{-1} \rVert 
  \leq \frac\Theta{\Re \lambda - \omega} \,.
\end{equation}
After reformulating \eqref{e.pde} in its \emph{mild formulation}
\begin{equation}
  \label{e.mild}
  U(t) = \e^{tA} U^0 + \int_0^t \e^{(t-s)A} \, B(U(s)) \, \d s \,,
\end{equation}
the contraction mapping theorem applies and we obtain local-in-time
well-posedness of our abstract semilinear evolution equation.

Let $\Phi^t$ denote the semiflow of \eqref{e.pde}, i.e.\ the map $U^0
\mapsto \Phi^t(U^0)$ such that $U(t) = \Phi^t(U^0)$ satisfies
\eqref{e.mild} with $U(0) = U^0$.  We sometimes write $\Phi(U^0,t)$ in
place of $\Phi^t$.  When $U^0 \in D(A)$, then $t \mapsto \Phi^t(U^0)$
is differentiable.

\subsection{Regularity of the semiflow}
\label{ss.diffSemiflow}

When $B=0$, then $t \mapsto \Phi^t(U)$ is $k$-times differentiable as
a map from $D(A^k)$ to $\kY$ for every $k \in \N$.  In this section,
we extend this result to semilinear evolution equations under suitable
assumptions on the nonlinearity $B$ and provide bounds on the
derivatives.

For $k \in \N_0$, we
define
\[
  \kY_k = D(A^k)
\]
endowed with the inner product
\begin{equation}
  \label{e.normYk}
  \langle U_1, U_2 \rangle_{\kY_k} 
  = \langle A U_1, A U_2 \rangle_{\kY_{k-1}} 
  + \langle  U_1, U_2 \rangle_{\kY_{k-1}} \,.
\end{equation}
Then  
\begin{equation}
  \label{e.normA}
  \norm{A}{\kY_{\ell+1} \to \kY_{\ell}} \leq 1
  \quad \text{and} \quad
  \norm{U}{\kY_{\ell}} \leq \norm{U}{\kY_{\ell+1}}
\end{equation}
for all $U \in \kY_{\ell+1}$.

Given $\delta>0$ and a hierarchy of open sets $\kD_\ell \subset
\kY_\ell$ for $\ell = 0, \ldots, L$ for $L \in \N$ with $\kD_0 \equiv
\kD$, we define $\kD_0^{-\delta} \equiv \kD^{-\delta}$ as in
\eqref{e.kD-delta} and, for $\ell=1, \dots, L$,
\[
  \kD^{-\delta}_\ell 
  \equiv \{ U \in \kD_\ell \colon 
           \dist_{\kY_\ell} (U, \partial \kD_\ell) > \delta \} \,.
\]
Then, by construction and due to \eqref{e.normA},
$\kB_{\delta}^{\kY_\ell}(U) \subset \kD_\ell$ for all $U \in
\kD^{-\delta}_\ell$ and $\ell=0, \ldots, L$.

Let $\kY_1$ be a Banach space continuously embedded into the Banach
space $\kY$.  Then $\kD_1 \subset \kY_1$ is called a $\delta_*$-nested
subset of $\kD \subset \kY$ if $\kD^{-\delta}_1 \subset \kD^{-\delta}
$ for all $\delta \in [0,\delta_*]$.  Furthermore we say that the
family $\kD_0, \dots, \kD_L$ is $\delta_*$-nested if
$\kD^{-\delta}_\ell \subset \kD^{-\delta}_{\ell-1}$ for all $\delta
\in [0,\delta_*]$ with $\delta_*>0$ and $\ell = 1,\ldots, L$.  For
example, the family $\kD_k = \kB_R^{\kY_k}(U^0)$ is $\delta_*$-nested
for every $\delta_* \in (0,R)$ and $U^0 \in \kY_L$.  However, an
arbitrary nested family $\kD_\ell \subset \kY_\ell$ may not be
$\delta_*$-nested for any $\delta_*>0$.

To state a differentiability result for higher time derivatives, we
need the following specific assumptions on the regularity of $B$ on
the scale $\kY_j$.  The same assumptions will also be required for the
convergence analysis of $A$-stable Runge--Kutta schemes in
Section~\ref{s.rk}.

\begin{itemize}
\item[(B1)] There exist $K \in \N_0$, $N \in \N$ with $N>K$, and a
sequence of $\delta_*$-nested $\kY_k$-bounded and open sets $\kD_k$
such that $B \in \kCb^{N-k}(\kD_k;\kY_k)$ for $k = 0, \dots, K$.
\end{itemize}

We denote the bounds on $B \colon \kD_k \to \kY_k$ and its derivatives
by constants $M_k$, $M_k'$, etc., for $k = 0, \dots, K$, and identify
$M=M_0$, $M'=M'_0$, and $\kD=\kD_0$.  In addition to the domains
$\kD_0, \dots, \kD_K$ defined in this assumption, we will sometimes
need to refer to $\kD_{K+1}$, which may be any $\delta_*$-nested
subset of $\kD_K$ which is bounded and open in $\kY_{K+1}$.

We note that nonlinear continuous operators $B\in \kC(\kD;\kY)$ do not
generally map closed bounded sets into closed bounded sets, see
Remark~\ref{r.unbounded} below.  However, the boundedness requirements
can always be met on balls inside the domain of $B$, i.e., if $B$ is
$\kC^n$ from some open set $\kD \subset \kZ$ to $\kZ$ then, by
continuity, for every $U^0 \in \kD$ there is some $R>0$ such that $B
\colon \kB_R(U^0) \subset \kD \to \kZ$ and its derivatives are
uniformly bounded so that $B \in \kCb^n(\kB_R(U^0);\kY)$.

\begin{remark} \label{r.unbounded} 
The existence of continuous unbounded nonlinear functionals on an
infinite-dimensional Banach space $\kX$ can be seen by the following
construction.  It is a standard result that there exists  a
sequence $x_j \in \kX$ such that $\lVert x_j \rVert = 1$ and $\lVert
x_j - x_k \rVert \geq 3/4$; on a Hilbert space, an orthonormal basis
will do.  Now let $h_j \in \kC(\kX,\R)$ have support on
$\kB^\kX_{1/4}(x_j)$ with $h_j(x_j)=1$.  Then $F$ defined by
\[
  F(x) = \sum_{j=0}^\infty j \, h_j(x)
\]
satisfies $F\in \kC(\kX,\R)$, since we have $h_j(x)=0$ for all but at
most one $j$.  But $F$ does not map the closed bounded set
$\kB_1^\kX(0)$ into a bounded set.
\end{remark}

Superposition operators of smooth functions $f \colon D\subset\R^d\to
\R^m$ on Sobolev spaces as occur in Examples~\ref{ex:SWE}
and~\ref{ex:NSE} above are bounded.
Indeed, for superposition operators we can construct 
$\delta_*$-nested domains such that condition (B1) holds, see Theorem~\ref{t.superposition}
and Sections~\ref{ss.swe} and~\ref{ss.nse} below.

Under assumptions (A0) and (B1), the semiflow $\Phi^t$ of
\eqref{e.pde} exists on each $\kY_k$.  In the following, we show that
a time derivative of order $\ell$ maps $\ell$ rungs down this scale of
Hilbert spaces.
 
\begin{theorem}[Regularity of the semiflow, local version] 
\label{t.local-diff} 
Assume \textup{(A0)} and \textup{(B1)}.  Choose $R \in (0, \delta_*]$
such that $\kD_K^{-R}\neq \emptyset$
 and pick $U^0 \in  \kD_K^{-R}$.  Let $R_* =
R/(2\Theta)$ with $\Theta$ from \eqref{e.etAEstimate}.  Then there is
$T_*=T_*(R,U^0)>0$ such that the semiflow $(U,t) \mapsto \Phi^t(U)$ of
\eqref{e.pde} satisfies
\begin{subequations}
\label{e.T-diff}
\begin{equation}
  \label{e.T-diff-gen}
  \Phi \in \bigcap_{\substack{j+k \leq N \\ \ell \leq k \leq K}}
  \kCb^{(\underline j, \ell)} (B_{R_*}^{\kY_K}(U^0) \times [0,T_*]; 
    \kB_R^{\kY_{k-\ell}}(U^0)) \,.
\end{equation}
In particular, 
\begin{equation}
  \label{e.T-diff-easy}
  \Phi \in \kCb^K (B_{R_*}^{\kY_K}(U^0) \times [0,T_*]; 
    \kB_R^\kY(U^0)) \,.
\end{equation}
\end{subequations}
The bounds on $\Phi$ and $T_*$ depend only on the bounds afforded by
\textup{(B1)}, \eqref{e.etAEstimate}, $R$, and $U^0$.
\end{theorem}

\begin{proof}
Writing $t=\tau T$ for some fixed $T>0$, we see that a solution to the
mild formulation \eqref{e.mild} is a fixed point of the map
\begin{equation}
  \label{e.PiFlow}
  \Pi(W;U,T)(\tau) = \e^{\tau T A} U +
  T \int_0^\tau \e^{(\tau-\sigma)TA} \, B(W(\sigma)) \, \d\sigma \,.
\end{equation}
This reformulation is useful because we want to quote the contraction
mapping theorem on a scale of Banach spaces, Theorem~\ref{t.cm-scale},
to prove the differentiability properties of $\Phi$ as claimed.  We
work on the scale $\kZ_j = \kCb([0,1]; \kY_j)$ and seek a fixed point
of $\Pi$ in $\kW_j = \kCb ([0,1]; \kB_R^{\kY_j}(U^0))$ for $j = 0,
\dots, K$, with parameter sets $\kU \equiv
\interior(\kB^{\kY_K}_{R_*}(U^0)) \subset \kX = \kY_K$ and $\kI =
(0,T_*)$.  Clearly, $\Pi$ maps $\kW_j\times \kU \times \kI$ into
$\kZ_j$.  To bound the range of $\Pi$, we estimate, for $j=0,\ldots,
K$,
\begin{align}
  & \norm{\Pi(W;U,T) - U^0}{\kY_j}
    \notag \\
  & \quad
    \leq \norm{\e^{\tau T A} U^0 - U^0}{\kY_j} + 
         \norm{\e^{\tau T A} (U - U^0)}{\kY_j}  +
         T \int_0^\tau 
         \norm{\e^{(\tau-\sigma)TA} \, B(W(\sigma ))}{\kY_j} \, \d\sigma
    \notag \\
  & \quad
    \leq \norm{\e^{\tau T A} U^0 - U^0}{\kY_j} + 
         \Theta \, \e^{\omega T} \, R_* +
         T \, \Theta \, \e^{\omega T} \, M_j \,. 
       \label{e.localEx}
\end{align}
With the choice $R_*={R}/{2\Theta}$, we observe that for sufficiently
small $T_*$ and all $T \in [0,T_*]$ the right hand side can be made
less than $R$ for $j=0,\ldots, K$ independent of $\tau \in [0,1]$.  We
can thus take the supremum over $\tau \in [0,1]$, which altogether
proves that $\Pi$ maps $\kW_j\times \kU \times \kI$ into $\kW_j$.
Condition (i) of Theorem~\ref{t.cm-scale} then follows from our
assumptions on $A$ and $B$.

Similarly, we estimate
\begin{equation}
  \label{e.flow_contraction}
  \norm{\D_W \Pi(W;U,T)}{\kE(\kCb([0,1];\kY_j))}
  \leq T \, \Theta \, \e^{\omega T} \, M'_j \,,
\end{equation}
so that $\Pi$ is a uniform contraction for all $U \in \kU$, $W \in
\kW$, and $T \in \kI = (0,T_*)$ with a possibly smaller value of
$T_*$.  Here we used that $B$ is at least $\kC^1$ on the highest rung
of the scale due to the requirement that $N>K$ in \textup{(B1)}.
Hence, condition (ii) of Theorem~\ref{t.cm-scale} is verified.
 
Theorem~\ref{t.cm-scale} then implies that the fixed point $W$ of
$\Pi$ satisfies
\[
  W \in \bigcap_{\substack{j+k \leq N \\ \ell \leq k \leq K}}
  \kCb^{(\underline j, \ell)} (B_{R_*}^{\kY_K}(U^0) \times [0,T_*]; 
  \kW_{k-\ell}) \,.
\]
To infer \eqref{e.T-diff-gen}, we recall that $\Phi^{\tau T}(U) =
W(U,T)(\tau)$; hence $\partial_U^m\partial_T^n \Phi^{T}(U) =
\partial_U^m\partial_T^n W(U,T)(1)$.  Finally, \eqref{e.T-diff-easy}
follows from Lemma~\ref{l.cm-scale}.
\end{proof}
  
\begin{remark} 
With the choice of norm \eqref{e.normYk}, the fundamental estimates in
this paper which carry named constants, in particular $\Theta$ in
\eqref{e.exptA-unif} and $\Lambda$, $c_\mS$ in
Lemma~\ref{l.RK.Lambda}, are the same on all $\kY_k$ for $k \in \N_0$
as these constants are norms of operators like $\e^{tA}$
 which   commute with $A$.  Thus, if $\kY_k$ for $k \in \N$
were endowed with a different, but equivalent set of norms, these and
consequent constants would need to be adopted and possibly become
dependent on the   rung.
\end{remark}

Theorem~\ref{t.local-diff} does not guarantee that the time of
existence of the solution can be chosen uniformly over $\kD$ or even
over $\kD^{-\delta}$ for some $\delta>0$.  The following theorem shows
that such uniformity can be obtained along with improved regularity
over bounded domains other than balls on the expense of requiring the
initial data to lie in a set one step up the scale.

In the following, define
\begin{equation}
  \label{e.Rk}
  R_{K+1} = \sup_{U \in \kD_{K+1}^{-\delta}} \norm{U}{\kY_{K+1}} \,.
\end{equation}

\begin{theorem}[Regularity of the semiflow, uniform version] 
\label{t.local-diff-unif} 
Assume \textup{(A0)} and \textup{(B1)}.  Choose $\delta \in (0,
\delta_*]$ small enough such that $\kD_{K+1}^{-\delta} \neq
\emptyset$.  Then there exists $T_*=T_*(\delta)>0$ such that the
semiflow $(U,t) \mapsto \Phi^t(U)$ of \eqref{e.pde} satisfies
\eqref{e.T-diff} with uniform bounds for all $U^0 \in
\kD_{K+1}^{-\delta}$, with $R=\delta$, and such that
\begin{subequations}
\begin{equation}
  \label{e.T-diff-gen-unif}
  \Phi \in \bigcap_{\substack{j+k \leq \mybf{N} \\ \ell \leq k \leq \mybf{K+1} }}
  \kCb^{(\underline j, \ell)} (\kD_{K+1}^{-\delta} \times [0,T_*]; 
  \kY_{\mybf{k-\ell}}) \,.
\end{equation}
In particular, when $N>K+1$,
\begin{equation}
  \label{e.T-diff-easy-unif}
  \Phi \in \kCb^{\mybf{K+1}}(\kD_{K+1}^{-\delta} \times [0,T_*]; \mybf{ \kD}) \,.
\end{equation}
\end{subequations}
The bounds on $\Phi$ and $T_*$ depend only on $\delta$ and on the
bounds afforded by \textup{(B1)}, \eqref{e.Rk}, and
\eqref{e.etAEstimate}.
\end{theorem}

\begin{proof}
We apply Theorem~\ref{t.local-diff} for each $U^0 \in
\kD_{K+1}^{-\delta}$ with $R=\delta$.  We note that in the proof of
Theorem~\ref{t.local-diff}, even in the case $K=0$, the guaranteed
time of existence $T_*$ cannot be chosen uniformly for $U^0 \in
\kD^{-\delta}$ because the first term on the right of
\eqref{e.localEx} cannot be made uniformly small.  However, we may
alternatively estimate, using \eqref{e.normA} and \eqref{e.Rk}, that
for $j=0,\ldots, K$
\begin{equation}
  \label{e.exptA-unif}
  \norm{\e^{\tau T A} U^0 - U^0}{\kY_j}
  \leq T \, \max_{t \in [0,T]} \norm{A\e^{t A} U^0 }{\kY_j}
  \leq T \,  \, \Theta \, \e^{\omega T} \, R_{1+j}  \,.
\end{equation}
Inserting this estimate into \eqref{e.localEx}, we see that we can
choose $T_*>0$ small enough such that $\Pi(\,\cdot\,; U, T)$ maps
$\kW_j=\kB_R^{\kZ_j}(U^0)$ into itself for all $U^0 \in
\kD_{K+1}^{-\delta}$ and $T \in [0,T_*]$.  Following the proof of
Theorem~\ref{t.local-diff}, we find that \eqref{e.T-diff-gen} holds
with uniform bounds for all $U^0 \in \kD_{K+1}^{-\delta}$ when
$R=\delta$, thereby implying
\begin{equation}
  \label{e.T-diff-gen-unif-1}
  \Phi \in \bigcap_{\substack{j+k \leq N \\ \ell \leq k \leq K}}
  \kCb^{(\underline j, \ell)} (\kD_{K+1}^{-\delta} \times [0,T_*]; 
  \kD_{k-\ell}) 
\end{equation}
with bounds which only depend on the bounds afforded by \textup{(B1)},
\eqref{e.etAEstimate}, \eqref{e.Rk}, and on $\delta$.  Next, we prove
that $\Phi$ maps into a space one step up the scale, namely
\begin{equation}
  \label{e.T-diff-gen-unif-Y1}
  A\Phi \in \bigcap_{\substack{j+k \leq N-1 \\ \ell \leq k \leq K}}
  \kCb^{(\underline j, \ell)} (\kD_{K+1}^{-\delta} \times [0,T_*]; 
  \kY_{k-\ell}) \,.
\end{equation}
 
Note that by \cite[Theorem 6.1.5]{Pazy}, a mild solution $U(t)$ of
\eqref{e.pde} satisfies $U(t) \in D(A)$ if $U(0) \in D(A)$ and $B \in
\kC^1(\kD,\kY)$; thus, the formal identity $\d W(\tau)/\d \tau = T \,
(A W(\tau T) + B(W(\tau)))$ for $W(\tau) = U(\tau T)$ holds true.
Hence, by applying $A$ to the fixed point equation \eqref{e.PiFlow}
and integrating by parts, we find
\begin{align*}
  A W(\tau) 
  & = \e^{\tau T A} A U
      + T \int_0^{\tau} A \e^{(\tau-\sigma)TA} \,
        B(W(\sigma )) \, \d \sigma 
      \notag \\
  & = \e^{\tau T A} \, (A U + B(U)) - B(W(\tau)) 
      \notag \\
  & \quad 
      + T \int_0^{\tau} \e^{(\tau-\sigma)TA} \, 
        \D B(W(\sigma)) (A W(\sigma) + B(W(\sigma))) \, \d\sigma \,. 
\end{align*}
This is a linear fixed point equation 
\begin{align}
  \tilde{W} = 
  \tilde\Pi(\tilde{W},U,T)(\tau) 
  & = \e^{\tau T A} \, (A U + B(U)) - B(W(\tau)) 
      \notag \\
  & \quad 
      + T \int_0^{\tau} \e^{(\tau-\sigma)TA} \, 
        \D B(W(\sigma)) (\tilde{W}(\sigma) + B(W(\sigma))) \, \d\sigma \,
  \label{e.AW_flow}
\end{align}
for $\tilde{W}(U,T) = AW(U,T)$.  We consider the fixed point equation
\eqref{e.AW_flow} for $\tilde W=AW$ with $\kW_j =\kB_{\kZ_j}^r(0)$ for
$j=0,\ldots, K$ with $r>0$ big enough such that $\tilde\Pi$ maps each
$\kW_j$ into itself.  Applying Lemma~\ref{l.cm-composition} (chain
rule on the scale of Banach spaces) and Lemma~\ref{l.DwPiClass} to the
right hand side of the fixed point equation \eqref{e.AW_flow}, we
verify once more the assumptions of Theorem~\ref{t.cm-scale} with $N$
replaced by $N-1$.  This yields \eqref{e.T-diff-gen-unif-Y1}.
 
It remains to be shown that we can translate improved spatial
regularity into differentiability in time by invoking the semilinear
evolution equation \eqref{e.pde}.  Due to \eqref{e.T-diff-gen-unif-1},
Lemma~\ref{l.cm-composition} implies that $B \circ \Phi$ is in the
same class \eqref{e.T-diff-gen-unif-Y1} as $A \Phi$ and, since
$\partial_t \Phi = A\Phi + B\circ \Phi$, so is $\partial_t \Phi$.
Combining this result, \eqref{e.T-diff-gen-unif-1}, and
\eqref{e.T-diff-gen-unif-Y1} via Lemma~\ref{l.technical1} implies
\eqref{e.T-diff-gen-unif}.

Finally, \eqref{e.T-diff-easy-unif} follows from
Lemma~\ref{l.cm-scale} with $K$ replaced by $K+1$.
\end{proof}
 
\begin{remark}
It is worth noting that, even though we find that $\Phi^t$ maps into
$\kY_{K+1}$, the proof, being based on the fixed point problem
\eqref{e.AW_flow}, requires $B$ to be defined only up to rung $K$.
The same pattern occurs when studying the Runge--Kutta numerical
time-$h$ maps in Section~\ref{s.rk}.
\end{remark}

\begin{remark}[Image of semiflow] \label{r.local-diff-unif}   
The proof of Theorem \ref{t.local-diff-unif} shows that, actually,
\[
  \Phi \in \bigcap_{\mybf{\substack{j+k \leq N \\ \ell \leq k \leq K+1 \\ 
                              (k,\ell) \neq (K+1,0)}}}
  \kCb^{(\underline j, \ell)} (\kD_{K+1}^{-\delta} \times [0,T_*]; 
   \mybf{ \kD_{k-\ell}}) \,.  
\]
\end{remark}

\begin{remark} \label{rem:HamPazy} 
If $A= -A^*$ is skew-symmetric on the Hilbert space $\kY$, as for the
nonlinear Schr\"odinger equation (see Section~\ref{ss.nse}), then $A$
is normal, $\i A$ is self-adjoint and, by Stone's theorem, generates a
unitary group $\e^{tA}$.  More generally, if $A$ is skew-symmetric up
to a perturbation which is bounded on all $\kY_k$, as for the
semilinear wave equation (see Section~\ref{ss.swe}), then $A$
generates a $C_0$ group $\e^{tA}$ on each $\kY_k$.  In both cases,
\eqref{e.etAEstimate} and \eqref{e.ResolventEst} may be replaced by
the following statement: There exist a constant $\omega$ with $\lvert
\Re (\spec A) \rvert \leq \omega$ and a constant $\Theta$ such that
for every $t \in \R$ and for every $\lambda \in \C$ with $\lvert \Re
\lambda \rvert > \omega$
\[
  \lVert \e^{tA} \rVert \leq \Theta \, \e^{\omega \lvert t \rvert},\quad
  \lVert (\lambda -A)^{-1} \rVert 
  \leq \frac\Theta{\lvert \Re \lambda \rvert - \omega} \,,
\]
see \cite{Pazy}.  Then the semiflow $\Phi^t$ is also a flow with
interval of existence $[-T_*,T_*]$ and regularity as specified in
Theorem~\ref{t.local-diff} and Theorem \ref{t.local-diff-unif}.
\end{remark}


\subsection{Superposition operators}
\label{ss.superpos}

To study the well-posedness of evolution equations such as
\eqref{e.swe} and \eqref{e.nse}, we need to consider superposition
operators $f \colon \mybf{\kG}_\ell \subset \kH_\ell \to \kH_\ell$ of
functions $f \colon \mybf{G} \subset \R^d \to \R^d$.  \mybf{This concept is
widely used; see, e.g., \cite{Henry,Pazy} for specific examples.  In
this section, we characterize superposition operators in sufficient
generality for later use.}
 
Let $I=[a,b] \subset \R$ be a bounded closed interval.  We write
$\kH_\ell(I; \R^d)$ to denote the Sobolev space of functions $u \colon
I \to \R^d$ whose weak derivatives up to order $\ell$ are contained in
$\kL_2(I;\R^d)$.
 
\begin{lemma}[{\cite{Adams}}] \label{l.algebra}
The space $\kH_{\ell}(\mybf{I};\R)$ is a topological algebra for every
 $\ell > 1/2$. Specifically, there exists a constant
$c=c(\ell)$ such that for every $u, v \in \kH_{\ell}(\mybf{I};\R)$ the
product $uv \in \kH_{\ell}(\mybf{I};\R)$ satisfies
\begin{equation}
  \label{eq:uvGm}
  \norm{uv}{\kH_{\ell}(\mybf{I},\R)}
  \leq c \, \norm{u}{\kH_{ \ell}(\mybf{I};\R)} \, 
            \norm{v}{\kH_{ \ell}(\mybf{I};\R)} \,.
\end{equation}
\end{lemma}
  
Armed with this result, we can characterize more general superposition
operators where a function $f \colon \mybf{G} \to \R^m$ for some open $\mybf{G}
\subset \R^d$ induces a mapping $u \mapsto f(u)$ between function
spaces.  The $k$th derivative of $f$ as a function on $\R^d$ is a
$k$-linear map on $\R^d$.  As such, it induces a $k$-linear
superposition operator between function spaces.  \emph{A priori}, it
is not clear whether the $k$th Fr\'echet derivative of the
superposition operator of $f$ equals the superposition operator of the
$k$th derivative of $f$ on $\R^d$.  The following lemma and theorem
provide a setting in which this is true, so that we use the symbol
$\D^k f$ for both these objects. 

\begin{lemma} \label{l.superposition}
Let $\mybf{G} \subset \R^d$ be open, let $f \in \kCb^N(\mybf{G};\R^m)$ for some $N
\in \N_0$, and set
\[
  \mybf{\kG} = \{ u \in \mybf{\kC}(\mybf{I};\R^d) \colon u(\mybf I) \subset \mybf{G} \} \,.
\]
Then $f \in \kCb^N(\mybf{\kG}, \kC(\mybf{I};\R^m))$ and the derivatives of $f$ as an
operator from $\kC(\mybf{I};\R^d)$  to $\kC(\mybf{I};\R^m)$
are the superposition operators of the
derivatives of $f$ as a function on $\R^d$.
\end{lemma}
 
\begin{proof}
We proceed iteratively for $n=0, \dots, N$.  The Taylor theorem with
integral remainder asserts that for fixed $z_0 \in \mybf{G}$
\begin{equation}
  \label{e.taylor}
   \biggl\vert 
     f(z) - \sum_{i=0}^n \frac{\D^i f(z_0)}{i!} \, (z-z_0)^i 
   \biggr\vert
   \leq \rho(z_0,z) \, \lvert z-z_0 \rvert^n
\end{equation}
(when $d>1$, $\D^i f$ is an $i$-linear map acting on the tensor
product $(z-z_0)^i$), where
\[
  \rho(z_0,z)
   = \frac{1}{n!} \, \max_{\theta \in[0,1]}
     \lvert \D^n f(z_0 + \theta(z-z_0)) - \D^n f(z_0) \rvert
\]
is continuous in $z_0,z \in \mybf{G}$ and uniformly continuous for $z_0, z
\in K$ whenever $K \subset \mybf{G}$ is compact.

We now fix $u_0 \in \mybf{\kG}$ and let $u\in \mybf{\kG}$.  
 Clearly,     $u_0(I)$ and $u(I)$  are compact subsets  of
$\mybf{G}$, so that, setting $z_0=u_0(x)$ and $z=u(x)$ in \eqref{e.taylor},
we may take the supremum over $x \in \mybf{I}$, thereby obtaining
\[
  \biggl\Vert
     f(u) - \sum_{i=0}^n \frac{\D^i f(u_0)}{i!} \, (u-u_0)^i 
  \biggr\Vert_{\kC(\mybf I;\R^m)}
  \leq \norm{\rho(u_0,u)}{\kC(\mybf I;\R^m)} \, 
       \norm[n]{u-u_0}{\kC(\mybf I;\R^d)} \,.
\]
Since $\rho(u_0,u_0)=0$, this proves that $f \in \kC^n(\mybf{\kG};
\kC(\mybf I;\R^m))$ and identifies the Fr\'echet derivative of order $n$ as
the superposition operator of the derivative of order $n$ on $\R^d$.

\mybf{Since $f \in \kCb(\mybf{G},\R^m)$, the set $f(\mybf{\kG})$ is a bounded subset of
$\kC(I;\R^m)$.  Moreover $f$ extends continuously to the boundary of
$\mybf{\kG}$ since $f \colon \mybf{G} \to \R^m$ does.}

To prove boundedness of $\D^k f$ as a map from $\mybf{\kG}$ to
$\kE^k(\kC(\mybf I;\R^d), \kC(\mybf I;\R^m))$ \mybf{for $k=1, \dots, N$}, we
employ its identification with the superposition operator of the
$k$-linear map $\D^k f$ on $\R^d$ and estimate
\begin{equation}
  \label{e.derivative-estimate}
  \norm{\D^k f(u)(u_1, \dots, u_k)}{\kC(\mybf I;\R^m)}
  \leq c \, \norm{\D^k f(u)}{\kC(\mybf I;\R^{md^k})} 
       \prod_{i=1}^k 
       \norm{u_i}{\kC(\mybf I;\R^d)}
\end{equation}
\mybf{for some $c>0$, noting that $\D^k f \in \kCb(\mybf{\kG}; \kC({I};
\R^{md^k}))$ by the argument for the case $k=0$.}
\end{proof}

The corresponding result on the Sobolev scale is as follows.

\begin{theorem} \label{t.superposition} 
Let $f \in \mybf{\kCb^N}(\mybf{G};\R^m)$ for some $N \in \N_0$ and open set
$\mybf{G} \subset \R^d$.  For each $\ell=1, \dots, N$, \mybf{let $\mybf{\kG}_\ell$
denote an $\kH_\ell$-bounded and open  subset of $\mybf{\kG} \cap \kH_\ell(\mybf{I};\R^d)$}
with $\mybf{\kG}$ as in Lemma~\ref{l.superposition}.  Then
\[
  f \in \kCb^N(\mybf{\kG}_1; \kL_2(\mybf{I};\R^m)) \cap
        \bigcap_{\substack{k + \ell \leq N \\ \ell \geq 1}}
          \kCb^k(\mybf{\kG}_\ell; \kH_\ell(\mybf{I};\R^m)) \,.
\]
The derivatives of $f$ as an operator on $\kH_\ell$ are the
superposition operators of the derivatives of $f$ as a function from
$\R^d$ to $\R^m$.
\end{theorem}

\begin{proof}
The statement $f \in \mybf{\kCb^N}(\mybf{\kG}_1; \kL_2(\mybf I))$ is a direct
consequence of Lemma~\ref{l.superposition} and the continuity of the
embeddings $\kH_1(\mybf I;\R^m) \subset\kC(\mybf I;\R^m)
\subset \kL_2(\mybf I;\R^m)$.  (The first inclusion is due to the
Sobolev embedding theorem.)

Next, we show that $f\in \mybf{\kCb^\ell}(\mybf{G};\R^m)$ is bounded
as an operator from $\mybf{\kG}_\ell$ to $\kH_\ell(\mybf I;\R^m)$ for
$\ell \mybf{= 1, \dots, N}$.  We proceed inductively in $\ell$.
Since, \mybf{for some $C_\ell>0$,}
\begin{equation}
  \label{e.H_ell-1_to_ell}
  \norm{w}{\kH_\ell} 
  \leq C_\ell \, (\norm{w}{\kH_{\ell-1}} + \norm{w_x}{\kH_{\ell-1}})
\end{equation}
for $w\in\kH_\ell(\mybf I)$, the inductive step is achieved by taking
$w = f(u)$ and showing that $\lVert \partial_x f(u)
\rVert_{\kH_{\ell-1}}$ is bounded over \mybf{$u \in \mybf{\kG}_\ell$}.
Indeed, when $\ell=1$, $\tnorm{f(u)}{\kL_{2}}$ is uniformly bounded in
$u \in \mybf{\mybf{\kG}_1}$ by the argument above, and there is a constant
$c_1>0$ such that
\[
  \norm{\partial_x f(u)}{\kL_{2}(\mybf{I};\R^m)} 
  \leq c_1 \, \norm{\D f(u)}{\kC(\mybf{I};\R^{dm})} \, 
              \norm{u_x}{\mybf{\kL_{2}}(\mybf{I};\R^d)} 
\]
is uniformly bounded for $u \in \mybf{\mybf{\kG}_1}$ by Lemma~\ref{l.superposition}.
We conclude that $f$ is bounded as map from $\mybf{\kG}_1$ to $\kH_1(\mybf{I};\R^m)$.  
When $\ell\geq 2$, applying the algebra inequality \eqref{eq:uvGm}
component-wise, we estimate
\[
  \norm{\partial_x f(u)}{\kH_{\ell-1}(\mybf{I};\R^m)} 
  \leq c_2 \, \norm{\D f(u)}{\kH_{\ell-1}(\mybf{I};\R^{dm})} \, 
              \norm{u_x}{\kH_{\ell-1}(\mybf{I};\R^d)} \,,
\]
for some constant $c_2>0$, where the right side is uniformly bounded
for $u \in \mybf{\mybf{\kG}_\ell}$ since $\tnorm{\D f(u)}{\kH_{\ell-1}
(\mybf{I},\R^{dm})}$ is uniformly bounded for $u \in \mybf{\mybf{\kG}_{\ell-1}}$ by
induction hypothesis.  Thus, by \eqref{e.H_ell-1_to_ell} with
$w=f(u)$, using the induction hypothesis once more, we obtain
boundedness of $f \colon \mybf{\kG}_\ell \to \kH_\ell(\mybf{I};\R^m)$.

To prove continuity and continuous differentiability of $f \colon
\mybf{\kG}_\ell \to \kH_\ell$, we introduce, for $k=0, \ldots, N-1$,
\[
  F_k(u_0,u) 
  = f(u) - \sum_{i=0}^k \frac{\D^i f(u_0)}{i!} \, (u-u_0)^i 
\]
and write $\partial_x F_k(u_0,u)$ in the form
\[
  \partial_x F_k
  =  \biggl[
      \D f(u) - \sum_{i=0}^k \frac{\D^{i+1}f(u_0)}{i!} \, (u-u_0)^i
    \biggr] \, \partial_x u 
    + \frac{\D^{k+1}f(u_0)}{k!} \, (u-u_0)^k \, 
      \partial_x(u-u_0) \,.
\]
When $\ell=1$, we estimate for every $k=0, \dots, N-1$, using
Lemma~\ref{l.superposition} and the Sobolev embedding theorem again,
that, \mybf{for $u \in \kG_\ell$},
\begin{align}
  \norm{\partial_x F_k}{\kH_{\ell-1}(\mybf{I};\R^m)}
  & \leq c_3 \,
    \biggl\lVert
      \D f(u) - \sum_{i=0}^k \frac{\D^{i+1}f(u_0)}{i!} \, (u-u_0)^i
    \biggr\rVert_{\kC(\mybf{I})} \, 
    \norm{u}{\kH_\ell} 
    \notag \\
  & \quad
    + c_3 \, \frac{1}{k!} \, \norm{\D^{k+1}f(u_0)}{\kCb(\mybf{I})} \, 
      \norm[k]{u-u_0}{\kC(\mybf{I})} \, 
      \norm{u-u_0}{\kH_\ell(\mybf{I})} \,
    \notag \\
  & \leq \sigma(u_0,u) \, \norm[k]{u-u_0}{\kH_\ell(\mybf{I})} \vphantom \int
  \label{e.Fx}
\end{align}
for some $\sigma \in \kC(\mybf{\kG}_\ell\times \mybf{\kG}_\ell;\R^+_0)$ with
$\sigma(u_0,u)=0$ and some constant $c_3>0$.  Moreover, since $f \in
\kC^N(\mybf{\kG}_1;\kL_2)$, there exists a function $\omega \in
\kC(\mybf{\kG}_1\times \mybf{\kG}_1;\R^+_0)$ with $\omega(u_0,u)=0$ such that
$\lVert F_k \rVert_{\kL_2} \leq \omega (u_0,u) \,
\norm[k]{u-u_0}{\kH_1}$.  Hence, \eqref{e.H_ell-1_to_ell} with
$w=F_k(u_0,u)$ implies $f \in \kC^{N-1} (\mybf{\kG}_1;\kH_1)$.
 
When $\ell \geq 2$, we obtain, by applying \eqref{eq:uvGm} recursively
and component-wise to the second term of $\partial_x F_k$, an estimate
as on the first and second line of \eqref{e.Fx} with
$\kH_{\ell-1}(\mybf{I})$ in place of $\kC(I)$ for every $k=0, \ldots,
N-\ell$.  Applying the induction hypothesis to both $f$ and $\D f$
shows, as before, that $f \in \kC^{N-\ell} (\mybf{\kG}_\ell;\kH_\ell)$ and
that its derivatives are the superposition operators of the
derivatives of $f$ as a function on $\R^d$.
 
Due to this identification, we can prove boundedness of $\D^k f$ as a
map from $\mybf{\kG}_\ell$ to $\kE^k(\kH_\ell(\mybf{I};\R^d),
\kH_\ell(\mybf{I};\R^m))$ by applying \eqref{eq:uvGm} recursively and
component-wise to $\D^k f(u)(u_1,\dots,u_k)$.  In this way we obtain
an estimate of the form \eqref{e.derivative-estimate} with $\kH_\ell$
in place of $\kCb$.  The bound is then achieved by noting that $\D^k f
\colon \mybf{\kG}_\ell \to \kH_\ell(\mybf{I};\R^{md^{k}})$ is a
bounded operator by the argument provided earlier in this proof for
$k+\ell \leq N$.

\mybf{Finally, we need to show that $\D^k f \colon \mybf{\kG}_\ell\to
\kE^k(\kH_\ell(I;\R^d), \kH_\ell(I;\R^m))$ extends continuously to the
boundary of $\mybf{\kG}_\ell$ when $k+\ell \leq N$.  For $k=0$ this follows
recursively from \eqref{e.Fx} and $f \in \kCb^N(\mybf{\kG};\kC(I))$ as
above.  Applying this result to $\D^k f \colon \mybf{\kG}_\ell \to
\kH_\ell({I};\R^{md^{k}})$ and using once again the identification of
derivatives of the superposition operator with the superposition
operators of the derivatives, we complete the proof.}
\end{proof}

\subsection{Example: the semilinear wave equation}
\label{ss.swe}

In the case of the semilinear wave equation \eqref{e.swe}, the
operators $A$ and $B$ are given by \eqref{e.AB}. 

\subsubsection{Periodic boundary conditions}
\label{sss.nwe-S1}

Since the Laplacian is diagonal in the Fourier representation, it is
easy to see that the spectrum of $A$ is given by $\spec A = \{ \i k
\colon k \in \Z\}$ and that the group generated by $\Q_0 A$ is unitary
on any
\[
  \kY_{\ell} = \kH_{\ell+1}(I;\R) \times \kH_{\ell}(I;\R) 
  \quad \text{for } \ell \in \N_0 \,.
\]
Here $\P_0$ is the spectral projection associated with eigenvalue $0$
and $\Q_0=\id-\P_0$.  Hence, $A$ generates a $C_0$-group on
$\kY_{\ell}$ and assumption (A0) is met.  The full group $\e^{tA}$,
however, is not unitary due to the secular term from the Jordan block
of $A$ when restricted to $\P_0 \kY_{\ell}$.

Assume that the nonlinearity $f$ of the semilinear wave equation
\eqref{e.swe} satisfies $f \in \mybf{\kCb^N}(\mybf{G};\R)$ for some
$N\in \N$ and some open set $\mybf{G} \subset \R$, \mybf{and let $\kD
= \kD_u \times \kD_v$ where $\kD_u$ is the set $\mybf{\kG}_1$ from
Theorem~\ref{t.superposition}} \mybf{and $\kD_v$ denotes an open
bounded subset of $\kL_2(I)$}.  Then, by
Theorem~\ref{t.superposition}, the nonlinearity $B$ satisfies
assumption (B1) on the scale defined above with $K<N$ if we
recursively define $\kD_k = \kD_{k-1}^{-\delta_*} \cap
\interior(\kB_R^{\kY_k}(0))$ for some $R>0$ with $\kD \subset
\kB_R^{\kY}(0)$ and choose $\delta_*>0$ small enough to ensure that
all $\kD_k$ are non-empty.  Hence, Theorems~\ref{t.local-diff}
and~\ref{t.local-diff-unif} give regularity of the flow of the
semilinear wave equation on the scale $\kY_k$ defined above.

\subsubsection{Neumann boundary conditions}
\label{sss.Neumann-NWE}

In the case of Neumann boundary conditions on $I = [0,\pi]$, we set
$\kY = \kH_1(I,\R) \times \kL_2(I,\R)$ as before; the operator $A$
then has the same spectrum and $\e^{t\Q_0 A}$ is again unitary.  In
this case, $\kY_k = \kH_{k+1}^\nb(I,\R) \times \kH_{k}^\nb(I,\R)$ with
\[
  \kH_{k}^\nb(I,\R) 
  = \{u \in \kH_{k}(I,\R) \colon u^{(2j+1)}(0) = u^{(2j+1)}(\pi)=0
    \text{ for }  j = 0, \ldots, \lfloor k/2 \rfloor - 1  \} \,.
\]
When $G \subset \R$ is open and $f \in \kCb^N(G;\R)$, assumption (B0)
holds as before on the open bounded set $\kD \subset \kY$ defined
above.  We claim that (B1) also holds for $K<N$.  To prove the claim,
we must show that $f$ maps $\kH_{k+1}^\nb(I,\R) \cap \kD_u$ into
$\kH_k^\nb(I,\R)$ for $k=0,\ldots, K$.  When $k=1$, no boundary
conditions need to be checked.  When $k=2$, we observe that
$(\partial_x f(u))(x) = f'(u(x)) \, u_x(x) =0$ for $x=0,\pi$ and $u
\in \kH_2^\nb \cap \kD_u$, so $f(u) \in \kH_1^\nb$.  Further, when $k
= 3, \dots, K$, all terms in the sum obtained from computing
$\partial_x^{2j+1} f(u)$ contain at least one odd derivative of $u$ of
order at most $2j+1$, so the boundary conditions remain satisfied.

\subsubsection{Dirichlet boundary conditions}
\label{sss.Dirichlet-NWE}

When endowed with Dirichlet boundary conditions, $A$ generates a
unitary semigroup. 
We take $I =[0,\pi]$ as before and set $\kY_k =\kH_{k+1}^0(I,\R)
\times \kH_k^0(I,\R)$, where
\begin{equation*}
  \kH_k^0(I,\R) = \{ u \in \kH_k(I,\R) \colon
  u^{(2j)}(0)  = u^{(2j)}(\pi) = 0 \text{ for } 
  j \in \N_0 \text{ with } 2j \leq k-1 \} \,.
\end{equation*}
Let $G \subset \R$ be open with $0 \in G$ and let $f \in\kCb^N(G,\R)$
as before.  Then condition (B0) is satisfied with $\kD= \kD_u \times
\kD_v$, as before.  Condition (B1) is satisfied if
$f^{(2j)}(0) = 0$ for $0\leq 2j \leq K-1$.

\mybf{When $f$ does not satisfy these boundary conditions,
\emph{necessary} conditions for the existence of time derivatives take
a complicated structure.  To see this, it suffices to consider
differentiability at $t=0$.  For $U'(0)$ to exist, we have the obvious
requirement that $v(0,0) = v( {0},\pi)=0$.  For $U''(0)$ to exist, the
non-homogeneous boundary condition $\partial_x^2 u(0,0) = \partial_x^2
u(0,\pi) = -f(0)$ needs to be satisfied.  For $U'''(0)$ to exist,
$\partial_x^2 v(0,0) = \partial_x^2 v(0,\pi)= 0$ must hold.  Finally,
for $U^{(4)}(0)$ to exist, a straightforward computation shows that
$\partial_x^4 u(0,x) + f''(0) \, u_x^2(0,x) = f'(0) \, f(0)$ must hold
at $x=0,\pi$.  This nonlinear boundary condition is difficult to
handle, and in this situation the space of initial conditions allowing
temporally smooth solutions is not an open set in a suitable Hilbert
space.  Therefore, we restrict our attention to nonlinearities $B$ of
the semilinear evolution equation \eqref{e.pde} which satisfy
condition (B1).}

\subsubsection{The semilinear wave equation on the line}
\label{sss.nwe-line}

When $I=\R$, we take $\kY = \kH_1(\R)\times \kL_2(\R)$.  Using the
Fourier transform, we verify that $\e^{t A}$ is unitary on $\kY$.
Lemma \ref{l.algebra} remains valid with $I = \R$, but the assertions
of Theorem~\ref{t.superposition} only hold true provided $0 \in G$ and
$f(0)=0$.  For example, when $f$ is a polynomial without constant term
and $\kD_k = \kB_R^{\kY_k}(0)$ for some $R>0$, then $B$ satisfies condition
(B1) and Theorem~\ref{t.local-diff} applies.

\subsubsection{A semilinear wave equation in an inhomogeneous
material} 
\label{sss.swe-nonhom}

Instead of \eqref{e.swe}, let us consider the non-constant coefficient
semilinear wave equation
\[ 
  \partial_{tt} u = \partial_x (a \, \partial_x u) + b \, u + f(u)
\]
where $a,b \in \kCb^N(I;\R)$ with $a(x) >0$ and $b(x) \leq 0$ for $x
\in I$.  For periodic boundary conditions and on the line, the setting
and conclusions of Sections~\ref{sss.nwe-S1} and~\ref{sss.nwe-line}
apply.  For Dirichlet boundary conditions on $I = [0,\pi]$, the spaces
$\kY_k$ also carry over from Section~\ref{sss.Dirichlet-NWE} and it is
straightforward to verify that (B1) is satisfied with $K=4$ provided
$f(0) = f''(0) = 0$ and $N>K$.

We remark that   the semilinear wave equation in inhomogeneous
media can, in principle, be solved numerically by splitting methods
(see the introduction for references).  Here, however, splitting
methods lose their advantage because the explicit computation of
$\e^{tA}$ is expensive for operators with non-constant coefficients.


\subsection{Example: the nonlinear Schr\"odinger
  equation}
\label{ss.nse}

We first consider periodic boundary conditions.  In this case, the
Laplacian is diagonal in the Fourier representation with eigenvalues
$-k^2$ and $A$ generates a unitary group on $ \kL_2(I;\C)$ and, more
generally, on $\kH_{\ell}(I;\C)$ with $\ell \in \N_0$.

In the notation of Section~\ref{ss.genSett}, we choose $\kY_{\ell} =
\kH_{2\ell+1}(I;\C)$.  Then (A0) is satisfied.  If the potential $V(u,
\overline u)$ satisfies $V \in \mybf{\kCb^{K+2+N}}(\mybf{G};\R)$ with
$K<N$ for some open subset $\mybf{G} \subset \R^2 \equiv \C$ then, by
Theorem~\ref{t.superposition}, the nonlinearity $B$ defined in
\eqref{e.nlsdefs} satisfies assumption (B1) with $\kD=\kG_1$ from
Theorem~\ref{t.superposition} and $\kD_k$ defined recursively
as for the semilinear wave equation (Section \ref{sss.nwe-S1}). 
 Therefore, Theorem~\ref{t.local-diff} and
Remark~\ref{rem:HamPazy} assert the existence of a flow $\Phi$ on
$\kY$ and specify its regularity.

In the case of Neumann boundary conditions, we choose $\kY_\ell =
\kH_{2\ell+1}^\nb(I;\C)$ with $I=[0,\pi]$ and $\kH_{2\ell+1}^\nb$
defined in Section~\ref{sss.Neumann-NWE}, so that (B1) is satisfied.

In the case of Dirichlet boundary conditions, we choose $\kY_\ell =
\kH_{2\ell+1}^0(I;\C)$ as defined in Section~\ref{sss.Dirichlet-NWE}.
Then (B1) is satisfied for any $V(u,\bar{u}) =v(|u|^{2})$ where $v \in
\kCb^{K+2+N}(\R^+_0;\R)$, in particular for the standard case where
$V(u) = \lvert u \rvert^4/2$.

On the line, $\kY_\ell = \kH_{2\ell+1}(\R;\C)$ and condition (B1) is
satisfied if, for example, $V(u)$ is a polynomial in $u_1=\Re (u)$ and
$u_2=\Im (u)$ with no linear term, so that $f(0) = \partial_{\bar u}
V(0,0)= 0$.

\begin{remark}
\mybf{While the setup in this section concern PDEs in one spatial
dimension, our results on superposition operators 
can be extended to ``nice'' $n$-dimensional
spatial domains since Lemma~\ref{l.algebra} holds on
$\kH_\ell(\Omega,\R^d)$ for $\ell > n/2$ \cite{Adams}, when, e.g.,
$\Omega \subset \R^n$ is a domain with smooth boundary or $\Omega =
\R^n$.  So we could also consider the nonlinear Schr\"odinger
equation on $\R^2$ and $\R^3$.}
\end{remark}

\begin{remark}[Inhomogeneous boundary conditions]
\mybf{We can treat inhomogeneous time-independent mixed linear boundary conditions of
the form $\operatorname{BC}(U)= g$ for the above examples by solving
$Av = 0$, $\operatorname{BC}(v)= g$ and then applying a Runge--Kutta
method to $U-v$.  This is equivalent to applying a Runge-Kutta method
to $U$ with boundary conditions $\operatorname{BC}(U)= g$, cf.\ the
discussion in \cite{LubOst93}.}
\end{remark}


\section{{\em A}-stable Runge--Kutta methods on \mybf{Hilbert} spaces}
\label{s.rk}

In this section, we first prove an abstract convergence result for
discretizations of evolution equations on \mybf{Hilbert} spaces.  Then, in
Section~\ref{ss.RKSett}, we introduce a class of   $A$-stable Runge--Kutta
methods which are well defined when applied to the semilinear PDE
\eqref{e.pde} under assumptions (A0) and (B0).  In
Section~\ref{ss.convergence}, we study the regularity of $A$-stable
Runge--Kutta methods under the additional condition (B1) and finally
apply the abstract convergence result to those schemes.


\subsection{\mybf{An abstract convergence theorem on Hilbert spaces}}
\label{ss.AbstrDiscr}

In this section we prove an abstract 
convergence result
for evolution equations on Hilbert spaces, Theorem~\ref{t.genConv}.
\mybf{Although this theorem is modeled after the basic local
convergence result for ODEs and there are a lot of results on the
convergence of time discretizations of specific PDEs in the
literature, see Section~\ref{s.intro}, we are not aware of any
result that is as general as this theorem.}

In the classical setting of ordinary differential equations $\dot y =
f(y)$, a one-step method $y^{n+1} = \psi^h(y^n)$ is of order $p$ if,
formally, $y(h) - \psi^h(y^0) = O(h^{p+1})$.  In other words, the
local error is controlled by the Taylor integral remainder of order
$p+1$.  It is then easy to show that the method is globally convergent
of order $p$; see, e.g., \cite{DdBor}.

The situation is more subtle in the case of a differential equation
\begin{equation}
  \label{eq:pde}
  \dot{U} = F(U)
\end{equation}
on a \mybf{Hilbert} space $\kX$: First, it is not clear whether the
time-$h$ map $\Psi^h$ associated with a given one-step method applied
to \eqref{eq:pde} is well defined as map from an open subset of $\kX$
to itself.  It depends on the equation and on the chosen one-step
method, and typically fails for explicit \mybf{Runge--Kutta} methods.
Second, even if $U \mapsto \Psi^h(U)$ is well defined and continuous,
its derivatives with respect to $h$ will usually fail to be defined on
the same set.  Thus, in order to control the Taylor remainder $U(h)
-\Psi^h(U^0)$ in the case of a discretization of a PDE \eqref{eq:pde},
we must consider the remainder as a map from a space $\kZ$ of high
regularity into a space $\kX$ of low regularity.  In this setting, the
usual proof that consistent one-step methods are convergent applies
under the following assumptions.

Let $\kX$ and $\kZ \subset \kX$ be \mybf{Hilbert} spaces, where $\kZ$
is continuously embedded in $\kX$ and let $\Psi^h$ be a one-step
discretization of \eqref{eq:pde} which is of classical order $p$.
Assume there exist sets $\kD_\kX \subset\kX$ and $\kD_\kZ \subset \kZ$
such that $\kD_\kX$ is open in $\kX$, $\kD_\kZ \subset \kD_\kX$, and
there exist constants $h_*>0$, \mybf{$\Theta_*>0$}, such that the
following hold.
\begin{itemize}
\item[\mybf{(C1)}] For fixed $h \in [0,h_*]$, the map $U \mapsto
\Psi^h(U)$ is $\kC^1(\kD_\kX;\kX)$. 
\mybf{Moreover, there exists a possibly $h$-dependent norm
$\tnorm{\,\cdot\,}{\kX,h}$ on $\kX$ with
\begin{equation}
  \label{e.h-depNorms}
  \norm{U}{\kX} \leq \norm{U}{\kX,h} 
  \leq \Theta_* \, \norm{U}{\kX} 
\end{equation}
for all $U \in \kX$ and $h \in [0, h_*]$}
such that 
\begin{equation}
  \label{e.dpsi-bound}
  \sup_{U \in \kD_\kX} 
  \norm{\D \Psi^h(U)} {\kE(\kX),\mybf h} = 1 + O(h) 
\end{equation}
for all $h \in [0, h_*]$.  \mybf{Here, $\norm{\cdot} {\kE(\kX),h}$
denotes the operator norm induced by $\norm{\cdot} {\kX,h}$.}

\item[\mybf{(C2)}] For fixed $U \in \kD_\kZ$, the map  $h \mapsto \Psi^h(U)$
  is in $\kC^{p+1}([0,h_*]; \kX)$, and
\begin{equation}\label{e.dhp}
 \sup_{\substack{U \in \kD_\kZ \\ h\in [0,h_*]}}
 \norm{\partial_h^{p+1} \Psi^h(U)}{\kX} < \infty.
 \end{equation}
\end{itemize}
\mybf{Condition (C1) can be seen as a stability condition, whereas condition
(C2) ensures consistency.}

\begin{theorem} \label{t.genConv}
In the setting above, fix $U^0 \in \kD_\kZ$ and suppose that there
exists a solution
\begin{equation}
  \label{e.solution}
  U \in \kC ([0,T]; \kD_\kZ) \cap \kC^{p+1} ([0,T]; \kD_\kX)
\end{equation}
to the initial value problem \eqref{eq:pde} for some $T>0$ with
$U(0)=U^0$.  Let $\Psi^h$ be a one-step discretization of
\eqref{eq:pde} of order $p \geq 1$; let $U^m = (\Psi^h)^m(U^0)$ denote
the associated numerical solution.

Then there exist constants $h_*>0$, $c_1$, and $c_2$, depending only
on $T$, the norm of $U$ in $\kC^{p+1}([0,T];\kX)$, $\dist_\kX(\{U(t)
\colon t \in [0,T] \}, \partial \kD_\kX)$, and on the constants from
\mybf{\eqref{e.h-depNorms}}, \eqref{e.dpsi-bound} and \eqref{e.dhp},
such that for every $h \in [0,h_*]$,
\[
  \norm{U^m - U(mh) }{\kX} \leq c_2 \, \e^{c_1 mh} \, h^p
\]
so long as $mh \leq T$. 
\end{theorem}

\begin{proof}
Since $\kD_\kX$ is open, there is some $\delta >0$ such that
$\kB^\kX_\delta(U(t)) \subset \kD_\kX$ for each $t \in [0,T]$.
Setting
\[
  E_m = \norm{U^m - U(mh)}{\kX,\mybf{h}} \,,
\]
we estimate, with $\Phi^t(U(s)) = U(t+s)$,
\begin{align*}
  E_{m+1}
  & \leq \norm{\Psi^h(U^m) - \Psi^h(U(mh))}{\kX,\mybf{h}}
         + \norm{\Psi^h(U(mh)) - \Phi^h(U(mh))}{\kX,\mybf{h}}
         \notag \\
  & \leq \sup_{\theta \in [0,1]} 
         \norm{\D \Psi^h(U(mh) + \theta \, (U^m-U(mh)))}{\mybf{\kE(\kX),h}} 
         \, E_m \notag \\
  & \quad
         + \frac{\mybf{\Theta_*} \, h^{p+1}}{(p+1)!} \, 
         \sup_{s \in [0,h]}
         \Bigl(
           \norm{\partial_s^{p+1} \Psi^s(U(mh))}{\kX} + 
           \norm{\partial_s^{p+1} \Phi^s(U(mh))}{\kX}
         \Bigr) 
         \notag \\
  & \leq \sup_{U \in \kD_\kX} \norm{\D \Psi^h(U)}{\mybf{\kE(\kX),h}} \, E_m 
         \notag \\
  & \quad
         + \frac{\mybf{\Theta_*} \, h^{p+1}}{(p+1)!} \, 
         \sup_{t \in [0,T]}
         \Bigl(
           \sup_{h\in [0,h_*]} \norm{\partial_h^{p+1} \Psi^h(U(t))}{\kX} 
           + \norm{\partial_t^{p+1} U(t)}{\kX} 
         \Bigr) 
         \notag \\
  & \leq (1 + c_1 \, h) \, E_m + c_3 \, h^{p+1} \,.
\end{align*}
The suprema in the estimate above are finite due to
\eqref{e.dpsi-bound}, \eqref{e.dhp} and \eqref{e.solution},
respectively, so long as $E_m < \delta$ since then, due to
\eqref{e.h-depNorms}, $U(mh) + \theta \, (U^m-U(mh)) \in \kD_\kX$.

Thus, since $E_0=0$,
\[
  E_m \leq c_3 \, h^{p+1} \, \frac{(1+h \, c_1)^m - 1}{h \, c_1}
      \leq \frac{c_3}{c_1} \, 
           \biggl( 1 + \frac{mh \, c_1}m \biggr)^m \, h^p
      \leq c_2 \, \e^{c_1 mh} \, h^p \,.
\]
Thus, we can choose $h_*$ small enough such that $E_m < \delta$ for
all $m \leq T/h_*$.  This concludes the proof.
\end{proof}

\begin{remark}
\mybf{The proof of Theorem~\ref{t.genConv} does not use any Hilbert
space structure, so that the result holds true when $\kX$ and $\kZ$
are Banach spaces.  However, condition (C1) is rather restrictive on
general Banach spaces, see Remark~\ref{rem:d.orderWrong} below and the
discussion in the introduction.}
\end{remark}

 
\subsection{Regularity of $A$-stable Runge--Kutta discretizations}
\label{ss.RKSett}

Applying an $s$-stage Runge--Kutta method to the semilinear evolution
equation \eqref{e.pde}, we obtain
\begin{subequations} \label{e.as-rk}
\begin{align}
  W   & = {U}^0 \, \1 + h \, {\a} \, 
          \bigl( {A} {W} + {B}(W) \bigr) \,,
          \label{eq:RKStagesIter} \\
  U^1 & = U^0 + h \, {\b}^T \, 
          \bigl( {A} W + {B}(W) \bigr) \,. 
          \label{eq:RKUpdate}
\end{align}
\end{subequations}
We write, with $U \in \kY$, 
\[
  \mybf{\1 U = 
  \begin{pmatrix}
    U \\ \vdots \\ U
  \end{pmatrix} \in \kY^s} \,, 
  \quad
  W = 
  \begin{pmatrix}
    W^1 \\ \vdots \\ W^s
  \end{pmatrix} \,,
  \quad
  {B}(W) = 
  \begin{pmatrix}
    B(W^1) \\ \vdots \\ B(W^s)
  \end{pmatrix} \,,
\]
where $W^1, \dots, W^s$ are the stages of the Runge--Kutta method,
\[
  (\a W)^i = \sum_{j=1}^s \a_{ij} \, W^j \,,
  \qquad
  \b^T W  = \sum_{j=1}^s \b_{j} \, W^j \,,
\]
and $A$ acts diagonally on the stages, i.e., $({A}W)^i = A W^i$ for
$i=1,\dots,s$.  
 
Written this way, it is not transparent that, under certain
conditions, this class of methods results in a well defined numerical
time-$h$ map $\Psi^h$ on a \mybf{Hilbert} space $\kY$.  A more
suitable form is achieved by rewriting \eqref{eq:RKStagesIter} as
\begin{equation}
  \label{eq:newRKStagesIter}
  W = \Pi(W; U, h) 
    \equiv (\id - h  {\a} {A} )^{-1} \, 
    (\1 U + h  {\a}  {B}(W)) \,.
\end{equation}
Noting that 
\begin{equation}
  (\id - h \a A)^{-1} = \id + h\a A \, (\id - h \a A)^{-1}
  \label{e.inv-ident}
\end{equation}
and inserting \eqref{eq:newRKStagesIter} into \eqref{eq:RKUpdate}, we
obtain 
\begin{align}
  \Psi^h(U) 
  & = U + h {\b}^T \, \bigl( A W(U,h) + B(W(U,h)) \bigr)
      \notag \\
  & = \mS(hA) U + h {\b}^T \, (\id - h \a A)^{-1} \, B(W(U,h)) \,,
  \label{eq:NewRKUpdate}
\end{align}
where $S$ is the so-called \emph{stability function}
\begin{equation}
  \label{e.SzRK}
  \mS(z) = 1 + z \b^T \, (\id - z \a)^{-1} \, \1 \,.
\end{equation}

We now make a number of assumptions on the method and its interaction
with the linear operator $A$.  First, we assume that the method is
$A$-stable in the sense of \cite{LubOst93}.  Setting $\C^- = \{ z \in \C
\colon \Re z \leq 0\}$, the conditions are as follows.
\begin{itemize}
\item[(RK1)] The stability function \eqref{e.SzRK} is bounded with
$\vert \mS(z) \rvert \leq 1$ for all $z \in \C^-$.
\item[(RK2)] The $s \times s$ matrices $\id - z \a$ are invertible for
all $z \in \C^-$.
\end{itemize}
Sometimes, we will also assume that $\a$ is invertible.

\begin{remark} \label{r.eigenvalues} 
The matrix $\id - z \a$ is invertible for all $z \in \C^-$ if and only
if $\a$ has no eigenvalues in $\C^- \setminus \{0\}$.  Its inverse is
then bounded uniformly for $z \in \C^-$ by a constant $\Lambda \geq 1$
(insert, in particular, $z=0$).
\end{remark}

\begin{remark}
In general, Runge--Kutta methods are called $A(\theta)$-stable for
some $\theta \in [0,\pi/2]$ if $\lvert \mS(z)\rvert \leq 1$ for all $z
\in \C$ with $\lvert \arg (-z) \rvert \leq \theta$; see, e.g.,
\cite{DdBor}.  A definition of $A(\theta)$-stability that requires, in
addition, invertibility of $\id - z \a$ was introduced by Lubich and
Ostermann \cite{LubOst93} in the context of parabolic equations; their
results also depend, to a large extent, on the invertibility of $\a$.
Thus, our assumptions can be described as $A(\theta)$-stability for
$\theta=\pi/2$ in the sense of \cite{LubOst93}.  Note that the
requirement $\theta=\pi/2$ arises as we include operators $A$ which
are not necessarily sectorial, \mybf{ but whose spectrum may, for
example, contain a strip about the imaginary axis, cf.\
Sections~\ref{ss.swe} and~\ref{ss.nse}.}
\end{remark}
 
\begin{example}
The implicit midpoint rule has stability function $\mS(z) = (1+
z/2)/(1-z/2)$, $s=1$, $\a_{11}= \frac{1}{2}$, and $\b_1 = 1$.
Conditions (RK1) and (RK2) are readily verified; moreover, $\a$ is
invertible.
\end{example} 

\begin{lemma} \label{l.glrk} Gauss--Legendre Runge--Kutta methods
satisfy \textup{(RK1)} and \textup{(RK2)} with $\a$ invertible.
\end{lemma}

\begin{proof}
Condition (RK1) is the classical notion of A-stability; it is proved
for Gauss--Legendre methods in \cite[Theorem~6.44]{DdBor}, for
example.

To verify condition (RK2), write $\mS(z) = \mP(z)/\mQ(z)$ as the
quotient of polynomials $\mP$ and $\mQ$ with no common roots.  We claim
that $\mQ(z) = \det (\id -z \a)$.  To see this, note first that $\det
(\id -z \a)$ arises naturally as the common denominator when solving
for the terms of an explicit rational expansion of $(\id - z \a)^{-1}$
by Cramer's rule; see the proof of \cite[Lemma~6.30]{DdBor}.  The
claim follows if we can show that the numerator does not have any
factor in common with $\det (\id -z \a)$.  Since $p=2s$ for
Gauss--Legendre methods \cite[Theorem~6.43]{DdBor}, $\deg \mQ \leq s$
and $\deg \mP \leq s$ for $s$-stage implicit Runge--Kutta methods
\cite[Lemma~6.30]{DdBor} and, generally, $p \leq \deg \mP + \deg \mQ$
\cite[Lemma~6.4]{DdBor}, we conclude that $\deg \mP=\deg \mQ=s$ so
that indeed $\mQ(z) = \det (\id -z \a)$.

Since, by (RK1), the rational function $\mS$ is nonsingular on $\C^-$,
all eigenvalues of $\a$ must lie outside of $\C^- \setminus \{0\}$.
This proves invertibility of $\id - z\a$ on $\C^-$, cf.\
Remark~\ref{r.eigenvalues}.  Finally, since $\mQ(z) = \det (\id - z
\a)$ has degree $s$, $\a$ must also be nonsingular.
\end{proof}

For the convergence analysis in Section~\ref{ss.convergence}, we need
the following additional assumption on the operator $A$ and on the
scheme.
\begin{itemize}
\item[(A1)] Assumption (A0) holds, and there exist constants
\mybf{$\omega_\mS, \Theta_\mS, h_*>0$} such that for all $h \in
[0,h_*]$ and $\mybf{n\in \N}$,
\begin{equation}
  \mybf{\norm{\mS^n(hA)}{\kY \to \kY}
  \leq \Theta_\mS \, \e^{\omega_\mS n h} \,.}
  \label{e.ShaEstimate} 
  \end{equation}
\end{itemize}
\mybf{If assumption (A1) holds,  we define, for $U \in \kY$,
\begin{equation}
  \label{e.h-depNorm}
  \norm{U}{\kY,h} \equiv \sup_{n\in \N_0} 
  \e^{-n\omega_\mS h} \, \norm{\mS^n(hA) U}{\kY} \,.
\end{equation}
Then $\tnorm{\,\cdot\,}{\kY,h}$ is equivalent to the $\kY$-norm in the
sense of \eqref{e.h-depNorms} with $\mybf{\Theta_*}=\Theta_\mS$.  Moreover,
there is some $\sigma>0$ such that
\begin{equation}
  \label{e.h-depNorm-II}
  \norm{\mS(hA)}{\kE(\kY),h} \leq \e^{\omega_\mS h} 
  \leq 1 + \sigma \, h
\end{equation}
for $h\in [0,h_*]$.}
 
\begin{remark} \label{rem:d.orderWrong} 
When an $A$-stable Runge--Kutta is applied to discretize a general
$C_0$-semigroup $\e^{tA}$ on a Banach space, estimate
\eqref{e.ShaEstimate} is in general false. 
A counterexample  
 is the
implicit midpoint rule applied to $A=\partial_x$ on $\kL_1(\R)$
\cite{HershKato}.  When $A$ is a sectorial operator, then
\eqref{e.ShaEstimate} is satisfied \cite{Palencia}. 
\end{remark}
 
\begin{remark}
In the time-continuous case discussed in Section~\ref{s.pde}, the
estimate corresponding to \eqref{e.ShaEstimate} \mybf{is
\eqref{e.etAEstimate}.  Note that,} by replacing the $\kY$-norm with
the equivalent norm $\lVert U \rVert = \sup_{t\geq 0}
\mybf{\e^{-\omega t}} \, \lVert \e^{tA} U \rVert_{\kY}$ the constant
$\Theta$ in \eqref{e.etAEstimate} becomes $1$, analogous to
\eqref{e.h-depNorm-II}.
\end{remark}

We state the following sufficient condition for (A1), which is often
satisfied in applications.

\begin{itemize}
\item[(A2)] Assumption (A0) holds, $\kY$ is a Hilbert space, and
\mybf{$A = A_\rn + A_\rb$ with $A_\rn$ normal and $A_\rb$ bounded as a
linear operator on each $\kY_0, \dots, \kY_K$.}
\end{itemize}
(Recall that an operator $A$ is normal if it is closed and $AA^* = A^*
A$.)  Condition (A2) implies that the non-normal part \mybf{$ A_\rb$}
of $A$ can be included with $B$ as it satisfies the sufficient
condition (B1).  Note that $ A_\rb$ is a bounded linear operator on
each $\kY_k$ if, for example, $A_\rb = \P A$ and $A_\rn = \Q A$ is
normal, where $\P$ is a spectral projector of $A$ onto a finite
dimensional subspace and $\Q = \id - \P$.
 
\begin{remark} \label{r.A2} In the case of the semilinear wave
equation, see Section~\ref{ss.swe}, assumption (A2) is satisfied with
\mybf{$A_\rb = \P_0A$, where $\P_0$ denotes} the spectral projection
corresponding to the eigenvalue $0$ of $A$.  In the case of the
nonlinear Schr\"odinger equation, see Section~\ref{ss.nse}, the
operator $A$ is normal, so that (A2) holds trivially.
\end{remark}

\begin{lemma} \label{l.RK.Lambda2} 
Assume that \textup{(RK1)} and \textup{(RK2)} hold and that $A$
satisfies conditions \textup{(A2)}.  Then \eqref{e.ShaEstimate} is
satisfied \mybf{with $\Theta_\mS=1$}.
\end{lemma}
 
Before we can prove Lemma~\ref{l.RK.Lambda2}, we need some technical
estimates on the operators which appear on the right of 
Eq.~\eqref{eq:NewRKUpdate} and Eq.~\eqref{eq:newRKStagesIter}.  In the
following, we denote $s$ copies of $\kY$ by $\kY^s$ and use the norm
\[
  \norm{W}{\kY^s} = \max_{j=1,\dots,s} \norm{W^j}{\kY} \,.
\]

\begin{lemma} \label{l.RK.Lambda} 
Assume \textup{(RK2)} and \textup{(A0)}.  Then, for $h_*>0$ small
enough, there exist $\Lambda \geq 1$ and $c_\mS \geq 1$ such that
\begin{subequations}
\label{e.RKShA.1}
\begin{gather}
  \norm{(\id - h\a A)^{-1}}{\kY^s \to \kY^s} \leq \Lambda 
  \label{e.RKShA.a} 
\intertext{and}
  \norm{h \a A (\id - h \a A)^{-1}}{\kY^s \to \kY^s} \leq 1 + \Lambda 
  \label{e.RKShA.b} 
\end{gather}
\end{subequations}
for all $h \in [0,h_*]$.  Moreover, for any $\ell, n, \in \N_0$,
\begin{subequations}
\begin{gather}
  (W, h) \mapsto (\id - h \a A)^{-1} W
  \text{ is a map of class }
  \kCb^{(\underline{n},\ell)}(\kY^s_\ell\times[0,h_*]; \kY^s)
  \,, \label{e.RKShA.A} \\
  (W, h) \mapsto h \a A (\id - h \a A)^{-1} W
  \text{ is a map of class }
  \kCb^{(\underline{n},\ell)}(\kY^s_\ell\times[0,h_*]; \kY^s)
  \,, \label{e.RKShA.B} \\
  \intertext{and}
  (W,h) \to h (\id-h\a A)^{-1}W
  \text{ is a map of class }
  \kCb^{(\underline{n},\ell+1)}(\kY^s_\ell \times [0,h_*];\kY^s)
   \,. \label{e.RKShA.B2}  
\end{gather}
  \label{e.RKShA.2}
\end{subequations}
\end{lemma}

\begin{remark}
\mybf{Estimates of the form \eqref{e.RKShA.1} were proved in
\cite{LubOst93} under the assumption that $A$ is sectorial.}
\end{remark}

\begin{proof}
Transforming $\a$ into Jordan normal form, we see that there exists a
constant $c=c(\a)$ such that
\[
  \norm{(\id - h \a A)^{-1}}{\kY^s \to \kY^s}
  \leq c \, \max_{i=1,\dots,k} 
       \norm[m_i]{(\id - h \lambda_i A)^{-1}}{\kY \to \kY}
\]
where $\lambda_1,\dots,\lambda_k$ are the eigenvalues of $\a$ with
algebraic multiplicities $m_1,\dots,m_k$.  

Hence, let $\lambda$ be one of the eigenvalues of $\a$; we know that
$\Re \lambda > 0$ due to assumption (RK2) and
Remark~\ref{r.eigenvalues}.  Referring to \eqref{e.ResolventEst}, we
estimate, for $h \geq 0$,
\[
  \norm{(\id - h \lambda A)^{-1}}{\kY \to \kY}
  \leq \frac1{\lvert h\lambda \rvert} \,
       \frac\Theta{\Re \frac{1}{h\lambda} - \omega}
  =    \frac\Theta{\frac{\lvert \lambda \rvert}{\Re\lambda} - 
                   \lvert h \lambda \rvert \, \omega} \,.
\]
Thus, the right hand bound is positive and finite for all $h \in
[0,h_*]$ provided that $h_*>0$ is small enough.  This proves estimate
\eqref{e.RKShA.a}.  Due to identity \eqref{e.inv-ident}, estimate
\eqref{e.RKShA.b} follows immediately. 

To prove continuity of the map $(\id - h \a A)^{-1}W \colon [0,h_*]
\to W$ for fixed $W \in \kY^s$, we proceed as follows.  Let
$\varepsilon > 0$.  Then for every $W_1 \in \kY_1^s$, $h,h' \in [0,h_*]$,
\begin{align}
  & \norm{(\id - h \a A)^{-1} W - (\id - h' \a A)^{-1} W}{\kY^s} 
    \notag \\
  & \quad 
    \leq \norm{((\id - h\a A)^{-1} - (\id - h'\a A)^{-1}) W_1}{\kY^s}
       + \norm{(\id - h \a A)^{-1} (W-W_1)}{\kY^s} 
    \notag\\
  & \quad \quad 
       + \norm{(\id - h' \a A)^{-1} (W-W_1)}{\kY^s} \notag \\
  & \quad
    \leq \norm{((\id - h \a A)^{-1} - (\id - h' \a A)^{-1})  W_1}{\kY^s}
         + 2\Lambda \, \norm{W-W_1}{\kY} \,,
  \label{e.LambdaC0}
\end{align}
where the second inequality is based on \eqref{e.RKShA.a}.  Now, since
$A$ is assumed to be densely defined and $\kY_1 = D(A)$, we can choose
$W_1$ so close to $W$ that the last term on the right is less than
$\varepsilon/2$.  Then, since $W_1 \in \kY_1^s$, there exists a
$\delta = \delta(W_1)$ such that the first term on the right is less
than $\varepsilon/2$ whenever $\lvert h - h' \rvert < \delta$.  This
proves continuity of $h \mapsto (\id - h \a A)^{-1}W$ on the interval
$[0,h_*]$.   

To complete the proof of \eqref{e.RKShA.2}, we must compute the
$h$-derivatives of the map \eqref{e.RKShA.A}.  Once we have shown
\eqref{e.RKShA.A}, estimate \eqref{e.RKShA.B} follows immediately via
\eqref{e.inv-ident}.  First,
\[
  \partial_h^\ell (\id - h \a A)^{-1} 
  = \ell! \, (\a A)^\ell \, (\id - h \a A)^{-\ell-1} \,.
\]
Using estimates \eqref{e.normA} and \eqref{e.RKShA.a}, and noting the
continuity of $h \mapsto (\id - h \a A)^{-1}W$ proved above,
\eqref{e.RKShA.A} follows.  Finally, noting that
\[
  \partial_h [ h (\id - h \a A)^{-1}] 
  = (\id - h \a A)^{-1} + h\a A(\id - h \a A)^{-2} 
  = (\id - h \a A)^{-2} \,,
\]
we obtain
\[
  \partial_h^\ell [ h (\id - h \a A)^{-1}] = 
  \partial_h^{\ell-1} (\id - h \a A)^{-2} =  
  \ell! \, (\a A)^{\ell-1} \, (\id - h \a A)^{-\ell-1} \,,
\]
which implies \eqref{e.RKShA.B2}.
\end{proof}

\begin{proof}[Proof of Lemma~\ref{l.RK.Lambda2}]
Recall that $\Re (\spec A) \leq \omega$ for some $\omega >0$ so that
the spectrum of $A-\omega$ is contained in $\C^-$.  Moreover, by
assumption (A2) we can split $A$ into a normal part \mybf{$A_\rn$} and a
bounded part \mybf{$A_\rb$}.  Now decompose $A = A_1 + A_2$ with $A_1 =
\mybf{A_\rn} - \omega$ and $A_2 = \mybf{A_\rb} +\omega$.  We now apply
the Runge--Kutta scheme to the linear problem $\partial_tU = AU$ in
two different ways.  First, we take the full $A$ and $B\equiv 0$;
second we take $A$ replaced by $A_1$ and $B(U)=A_2 U$.  Since the
respective numerical time-$h$ maps given by \eqref{eq:NewRKUpdate}
must be the same, we obtain the identity
\begin{subequations}
  \label{e.sha}
\begin{gather}
  \mS(hA) = \mS(h A_1) + h {\b}^T \, (\id - h \a A_1)^{-1} \, A_2 W
  \label{e.sha.a}
\intertext{where} 
   W = (\id - h  {\a} {A_1} )^{-1} \, (\1 + h \a A_2 W) \,.
  \label{e.sha.b}
\end{gather}
\end{subequations}
Rewrite \eqref{e.sha.b} as $W= MW + G$ with
\[
  M = (\id - h  {\a} {A_1} )^{-1} \,  h  {\a}  A_2 
  \qquad \text{and} \qquad 
  G = (\id - h  {\a} {A_1} )^{-1} \, \1 \,.
\] 
Since $A_2$ is bounded and, by Lemma~\ref{l.RK.Lambda}, $(\id - h {\a}
{A_1} )^{-1}$ is uniformly bounded for $h \in [0,h_*]$, the matrix $M$
has norm smaller than one for $h \in [0,h_*]$ with some possibly
smaller $h_*>0$.  Consequently, we can solve for $W= (\id -M)^{-1}G$,
whence the second term on the right of \eqref{e.sha.a} is $O(h)$ in
the norm of $\kE(\kY)$.

As $A_1$ is normal with $\spec A_1 \subset \C^-$, we have, referring
to (RK1),
\[
  \norm{\mS(h A_1)}{\kY \to \kY} \leq 
  \sup_{\lambda \in \spec A_1} \lvert \mS(h \, \lambda) \rvert
  \leq 1 \,.
\]
Altogether, this proves \mybf{that there exists $\sigma>0$ such that
$\norm{\mS(h A)}{\kY \to \kY} \leq 1+ \sigma \, h$ for $h\in [0,h_*]$.
This in turn implies \eqref{e.ShaEstimate} with $\Theta_\mS=1$.}
\end{proof}

\mybf{Next, we describe the differentiability properties of $\mS(hA)$
which will be needed later on.  }
 
\begin{lemma} \label{l.RK.ShA} 
Assume \textup{(RK2)}, \textup{(A0)}, and either that the Runge--Kutta
matrix $\a$ is invertible or that \textup{(A1)} holds.  Then there
exist $h_*>0$ and $c_\mS \geq 1$ such that for all $h \in [0,h_*]$,
\begin{equation}
  \norm{\mS(hA)}{\kY \to \kY} \leq c_\mS 
  \label{e.RKShA.c} 
\end{equation}  
and, for all $\ell,n \in \N_0$, 
\begin{equation}
  (U, h) \mapsto \mS(hA) U
  \text{ is a map of class }
  \kCb^{(\underline{n},\ell)}(\kY_\ell\times[0,h_*]; \kY) \,.
  \label{e.RKShA.C}
  \end{equation}  
\end{lemma}

\begin{proof}
First, \eqref{e.RKShA.c} clearly holds when \textup{(A1)} holds.  To
prove \eqref{e.RKShA.c} when $\a$ is invertible, we estimate, using
\eqref{e.SzRK} and \eqref{e.RKShA.b},
\[
  \norm{\mS(hA)}{\kY \to \kY}
  \leq 1 + s \, \norm{\b}{} \,\norm{\a^{-1}}{} \, (1+\Lambda) 
  \equiv c_\mS \,.
\]
Next, we show that $\mS(hA)U \colon [0,h_*] \to \kY$ is continuous for
every $U \in \kY$ as in the proof of Lemma~\ref{l.RK.Lambda},
replacing $(\id - h \a A)^{-1}$ by $\mS(hA)$, $\Lambda$ by $c_\mS$,
and $\kY^s$ by $\kY$ in \eqref{e.LambdaC0}.  This proves
\eqref{e.RKShA.C} for $\ell=0$.

Finally, to prove \eqref{e.RKShA.C} for $\ell \in \N$, we note that,
due to \eqref{e.RKShA.B2}, the map $(W,h) \mapsto hA(\id-h\a A)^{-1}W$
is of class $\kCb^{(\underline{n},\ell)}(\kY^s_{\ell}\times[0,h_*];
\kY^s)$; the claim then follows directly from the definition of $\mS$
in \eqref{e.SzRK}.
\end{proof}

In Theorem~\ref{t.local-diff}, we studied differentiability in time of
the semiflow $\Phi^t$ of \eqref{e.pde}.  An analogous result holds for
differentiability of the discretization $\Psi^h$ of \eqref{e.pde} in
the step size $h$.

\begin{theorem}[Existence and regularity of numerical method, local version] 
\label{t.h-diff} 
Assume that the semilinear evolution equation \eqref{e.pde} satisfies
conditions \textup{(A0)} and \textup{(B1)}, and apply a Runge--Kutta
method subject to condition \textup{(RK2)} to it.  Moreover, assume
that \textup{(A1)} holds or that the Runge--Kutta matrix $\a$ is
invertible.  Choose $R \in (0,\delta_*]$ such that $\kD^{-R}_K \neq
\emptyset$ and pick $U^0 \in \kD^{-R}_K$.  Let
$R_*=R/(2\max\{c_\mS,\Lambda\})$ with $c_\mS$ from \eqref{e.RKShA.c}
and $\Lambda$ from \eqref{e.RKShA.1}.  Then, for sufficiently small
$h_*>0$, there exists a unique stage vector $W$ and numerical time-$h$
map $\Psi(U,h)=\Psi^h(U)$ which satisfy
\begin{equation}
\label{e.h-diff-gen}
  W^i,\Psi \in \bigcap_{\substack{j+k \leq N \\ \ell \leq k \leq K}}
  \kCb^{(\underline j, \ell)} (\kB_{R_*}^{\kY_K}(U^0) \times [0,h_*]; 
  \kB_{R}^{\kY_{k-\ell}}(U^0))
\end{equation}
for $i = 1, \ldots, s$.  In particular,
\begin{equation} 
  \label{e.h-diff} 
  W^i, \Psi \in\kCb^K (B_{R_*}^{\kY_K}(U^0) \times [0,h_*]; 
  \kB_{R}^{\kY}(U^0)) 
\end{equation}
for $i=1, \dots, s$.  The bounds on $W$, $\Psi$ and $h_*$ depend only
on the bounds afforded by \textup{(B1)}, \eqref{e.RKShA.1},
\eqref{e.RKShA.c}, on the coefficients of the method, $R$, and $U^0$.
If, in addition, \textup{(A1)} holds, then there exists a constant
$\sigma_\Psi$, such that for $h \in [0,h_*]$ with a possibly smaller
choice of $h_*>0$,
\begin{equation}
  \label{e.DPsi}
  \sup_{U \in B_{R_*}^{\kY}(U^0)}
  \norm{\D \Psi^h(U)}{{\mybf{\kE(\kY),h}}} \leq 1 +\sigma_\Psi h \,,
\end{equation}
where the norm on the left is defined by \eqref{e.h-depNorm}
and $h_*$ and $\sigma_\Psi$ depend only on the above quantities and on the
constants in \textup{(A1)}.  
\end{theorem}

Note that statement \eqref{e.h-diff} for $\Psi$ is analogous to
\eqref{e.T-diff-gen} for the semiflow $\Phi$.
 
\begin{proof}
We apply the contraction mapping theorem on a scale of Banach spaces,
Theorem~\ref{t.cm-scale}, to the map from \eqref{eq:newRKStagesIter},
\[
  \Pi(W; U, h) \equiv (\id - h \a A)^{-1} \, \1 U
  + h \a \, (\id - h \a A)^{-1} \, B(W)
\]
with $u=U$, $w=W$, and $\mu=h$ on the scale $\kZ_j = \kY_j^s$ for
$j=0, \dots K$.  We further identify $\kX=\kY_K$, $\kW_j = \kB_R(\1
U^0) \subset \kY_j^s$, $\kI = (0, h_*)$, and $\kU = \interior
\kB_{R_*}(U^0) \subset \kY_K$.  To verify condition (i) of
Theorem~\ref{t.cm-scale}, we note that Eq.~\eqref{e.RKShA.A} of
Lemma~\ref{l.RK.Lambda} asserts that the map $(U,h) \mapsto (\id - h
\a A)^{-1} \, \1 U$ is, in particular, of class
\[
  \bigcap_{\substack{j+k \leq N \\ \ell \leq k \leq K}}
  \kCb^{(\underline j, \ell)} (\kD_K \times [0,h_*]; \kY_{k-\ell}^s) \,.
\]
The differentiability assumptions on $B$ from (B1) are precisely such
that the map $(W,h) \mapsto h \a \, (\id - h \a A)^{-1} \, B(W)$ is of
the same class. 

First, we show that $\Pi (\, \cdot \,; U, h)$ maps $\kW_j$,
$j=0,\ldots, K$, into itself for fixed $U \in \kU$ and $h \in [0,h_*]$
with appropriate $h_*>0$.  We begin by taking $h_*$ as in
Lemma~\ref{l.RK.Lambda} and estimate, for $W \in \kW_j$,
\begin{align}
  \norm{\Pi(W;U,h) - \1 U^0}{\kZ_j} 
  & \leq \norm{(\id - (\id - h {\a} {A} )^{-1}) \1 U^0}{\kZ_j}
    \notag \\
  & \quad + \norm{(\id - h {\a} {A} )^{-1}}{\kY_j^s \to \kZ_j} \,
            \norm{U-U^0}{\kY_j}
    \notag \\
  & \quad + h \, \norm{(\id - h \a A)^{-1}\a}{\kY_j^s \to \kZ_j} \, 
            \norm{B(W)}{\kZ_j}
     \notag \\
  & \leq \norm{(\id - (\id - h {\a} {A} )^{-1})\1 U^0}{\kZ_j} 
         + \Lambda \, R_* +  h \, \Lambda \, \norm{\a}{} \, M_j \,,
  \label{e.helpRK}
\end{align}
where, in the last step, we have used \eqref{e.RKShA.a} from
Lemma~\ref{l.RK.Lambda} and $M_j$ is the bound on $B$ from condition
(B1).  Since, again by Lemma~\ref{l.RK.Lambda}, the map $h \mapsto
(\id - h {\a} {A})^{-1}W$ is continuous on each $\kZ_j$, we can
possibly shrink $h_*$ such that the right hand side of
\eqref{e.helpRK} is less than $R$.  This proves that $\Pi (\, \cdot
\,; U, h)$ maps $\kW_j$ into itself and implies condition (i) of
Theorem~\ref{t.cm-scale}.

Next, for $j=0,\ldots, K$,
\begin{align}
  \norm{\D_W \Pi(W;U,h)}{\kZ_j \to \kZ_j} 
  & \leq h \, \norm{(\id - h {\a} {A})^{-1}\a}{\kZ_j \to \kZ_j} \,
         \norm{\D B(W)}{\kZ_j \to \kZ_j} 
    \notag \\
  & \leq h \, \Lambda \, \norm{\a}{} \, M'_j \,.
    \label{e.contraction}
\end{align}
Thus, by possibly shrinking $h_*$ again, the right hand bound can be
made less than $1$.  This proves that $\Pi (\, \cdot \,; U, h)$ is a
contraction on $\kB_{R}^{\kY^s}(\1 U^0)$ uniformly for $U \in
\kB_{R_*}^{\kY} (U^0)$ and $h \in [0,h_*]$.  Here we used that $B$ is
at least $\kC^1$ on the highest rung $\kY_K$ of the scale since, in
condition (B1), we require $N>K$.  This verifies condition (ii) of
Theorem~\ref{t.cm-scale}.

Theorem~\ref{t.cm-scale} then applies and asserts the existence of a
fixed point
\[
  W \in \bigcap_{\substack{j+k \leq N \\ \ell \leq k \leq K}}
  \kCb^{(\underline j, \ell)} (B_{R_*}^{\kY_K}(U^0) \times (0,h_*); 
  \kB_R^{\kY_{k-\ell}^s}(U^0\1)) \,.
\]  
Assertion \eqref{e.h-diff} for the $W^i$ then follows from
Lemma~\ref{l.cm-scale}.

To prove the corresponding estimates for $\Psi^h$, note that by
\eqref{eq:NewRKUpdate}, condition (B1), Lemma~\ref{l.RK.Lambda}, and
Lemma~\ref{l.RK.ShA} we can adapt $h_*>0$ such that for $j=0,\ldots,
K$,
\begin{align*}
  \norm{\Psi^h(U) - U^0}{\kY_j} 
  & \leq \norm{\mS(hA)(U-U^0)}{\kY_j} + 
         \norm{\mS(hA)U^0-U^0}{\kY_j} + 
         h \norm{b}{} \Lambda M_j \\
  & \leq R/2 + \norm{\mS(hA)U^0-U^0}{\kY_j} + 
         h \norm{b}{} \Lambda M_j \leq R \,.
\end{align*}
Further, the first term of \eqref{eq:NewRKUpdate} is of class
\eqref{e.h-diff} by Lemma~\ref{l.RK.ShA}.  For the second term of
\eqref{eq:NewRKUpdate}, we note that the map $\Sigma$ defined as
\[ 
  \Sigma(W,U,h) = h (\id - h\a A)^{-1}B(W) \,,
\]
  satisfies
\[
  \Sigma \in \bigcap_{\substack{i+j+k \leq N \\ \ell \leq k \leq K}}
  \kCb^{(\underline i, \underline j, \ell)} 
  ((\kD_k)^s \times \kY_K \times \kI; \kY^s_{k-\ell}) \,.
\]
Lemma~\ref{l.cm-composition} then implies \eqref{e.h-diff} for $\Psi$.
Assertion \eqref{e.h-diff} for $\Psi$ then follows from
Lemma~\ref{l.cm-scale}.
 
Finally, differentiating \eqref{eq:NewRKUpdate} and taking the
operator norm on $\kY$, we obtain
\begin{align*}
 & \norm{\D \Psi^h(U)}{\mybf{\kE(\kY),h}}
   \leq \norm{\mS(hA)}{\kE(\kY),\mybf{h}} \notag \\
  & \quad + h s \,\mybf{\Theta_\mS} \,  \norm{\b}{} \,
         \norm{(\id - h \a A)^{-1}}{\kE(\kY^s)} \,
         \norm{\D B(W)}{\kE(\kY^s) } \,
         \norm{\D_U W(U,h)}{\kE(\kY,\kY^s) }
         \notag \\
  & \leq (1 + \sigma \, h) + h \, s \, \mybf{\Theta_\mS} \, 
         \norm{\b}{} \, \Lambda \, M_0' \,
         \norm{W}{\kCb^{(\underline{1},0)}(\kB_{R_*}^{\kY}(U^0) \times
                  [0,h_*];\kY^s)}
         \notag \\
  & \equiv 1+ \sigma_\Psi \, h \,,
\end{align*}
where we use \mybf{\eqref{e.h-depNorms}, \eqref{e.h-depNorm-II}} and
\eqref{e.RKShA.a}, and refer to \eqref{e.h-diff} for the bound on $W$.
This proves \eqref{e.DPsi}.
\end{proof}
 
While this theorem gives an existence and regularity result for the
numerical time-$h$ map $\Psi^h$, it does not yield control over the
maximum step size $h_*$ when we want to define $\Psi^h$ on a general
open bounded domain.  We address this issue in the following theorem
which is the discrete time analogue of
Theorem~\ref{t.local-diff-unif}.
 
\begin{theorem}[Existence and regularity of numerical method, uniform
version] \label{t.h-diff-unif} 
Let the semilinear evolution equation \eqref{e.pde} satisfy conditions
\textup{(A0)} and \textup{(B1)} and apply a Runge--Kutta method
subject to condition \textup{(RK2)} to it.  Moreover, assume
\textup{(A1)} or that the Runge--Kutta matrix $\a$ is invertible.
Choose $\delta \in (0, \delta_*]$ small enough such that
$\kD_{K+1}^{-\delta}$ is non-empty.  Then there exists $h_{*}>0$ such
that \eqref{e.h-diff-gen} and \eqref{e.h-diff} and, under assumption
\textup{(A1)}, \eqref{e.DPsi} hold with bounds uniform for $U^0 \in
\kD_{K+1}^{-\delta}$ with $R=\delta$.  Moreover, the stage vector
$W(U,h)$ satisfies
\begin{subequations}
\begin{equation}
  \label{e.h-diff-unif-a}
  W \in \bigcap_{\substack{j+k \leq N \\ \ell \leq k \leq K\mybf{+1}}}
  \kCb^{(\underline j, \ell)} (\kD_{K+1}^{-\delta} \times [0,h_*]; 
  \kY_{\mybf{k-\ell}}^s) 
\end{equation}
and, if $\a$ is
invertible, the numerical time-$h$ map $\Psi(U,h)=\Psi^h(U)$ satisfies
\begin{equation}
  \label{e.h-diff-unif-b}
  \Psi \in \bigcap_{\substack{j+k \leq N \\ \ell \leq k \leq K\mybf{+1}}}
  \kCb^{(\underline j, \ell)} (\kD_{K+1}^{-\delta} \times [0,h_*];
  \kY_{\mybf{k-\ell}}) \,.
\end{equation}
\end{subequations}
In particular,  when $N>K+1$,
\begin{subequations}
\begin{equation} 
  \label{e.h-diff-unif-A}
  W^j \in\kCb^{K\mybf{+1}} (\kD_{K+1}^{-\delta} \times [0,h_*]; \mybf{\kD}),\quad
  j=1,\ldots, s,
\end{equation}
and, if addition $\a$ is invertible,
\begin{equation}
  \label{e.h-diff-unif-B}
  \Psi \in\kCb^{K\mybf{+1}} (\kD_{K+1}^{-\delta} \times [0,h_*]; \mybf{\kD}) \,.
\end{equation} 
\end{subequations}
The bounds on $W$, $\Psi$ and $h_*$ depend only on the bounds afforded
by \textup{(B1)}, \eqref{e.Rk}, \eqref{e.RKShA.1}, \eqref{e.RKShA.c},
on the coefficients of the method, and on $\delta$.
\end{theorem}

\begin{proof}
Let $R=\delta$.  We apply Theorem~\ref{t.h-diff} for each $U^0 \in
\kD^{-\delta}_{K+1}$.  Note that for $j=0,\ldots, K$,
\begin{align*}
  \norm{(\id - (\id - h {\a} {A} )^{-1}) \1 U^0}{\kY^s_j}
  & \leq h \max_{s \in [0,h]} 
         \norm{\a A (\id - s {\a} {A} )^{-2} \1 U^0}{\kY^s_j}
    \notag \\
  & \leq h \, \Lambda^2 \, \norm{\a}{} \,  R_{K+1} \,. 
\end{align*}
Inserting this estimate into \eqref{e.helpRK}, we see that we can
choose $h_*>0$ small enough such that $\Pi( \, \cdot \,; U,h)$ maps
$\kW_j(U^0) \equiv \kB_{R}(\1 U^0) \subset \kY^s_j$ into itself for
$j=0,\ldots, K$, and, from \eqref{e.contraction}, such that $\Pi$ is a
contraction on $\kW_j(U^0)$ uniformly for $U^0 \in
\kD_{K+1}^{-\delta}$, $U \in \kB^{\kY_{K+1}}_{R_*}(U^0)$, and $h \in
[0,h_*]$, where $R_* = R/(2\max\{c_\mS,\Lambda\})$.  As in the proof
of Theorem~\ref{t.h-diff}, we find that \eqref{e.h-diff-gen},
\eqref{e.h-diff} and \eqref{e.DPsi} hold with uniform bounds in $U^0
\in \kD_{K+1}^{-\delta}$, and that
\begin{equation}
  \label{e.unifReg}
  W^i,\Psi \in \bigcap_{\substack{j+k \leq N \\ \ell \leq k \leq K}}
  \kCb^{(\underline j, \ell)} (\kD_{K+1}^{-\delta} \times [0,h_*]; 
  \kD_{k-\ell})
\end{equation}
for $i=1,\ldots, s$.

To prove that $W$ actually maps into a space one step up the scale, we
show that
\begin{equation}
  \label{e.unifReg-Y1}
  AW^i  \in \bigcap_{\substack{j+k \leq N \\ \ell \leq k \leq K}}
  \kCb^{(\underline j, \ell)} (\kD_{K+1}^{-\delta} \times [0,h_*]; 
  \kY_{k-\ell})
\end{equation}
for $i=1,\ldots, s$.  We apply $A$ to \eqref{eq:newRKStagesIter}, so
that
\begin{equation}
  A W = A(\id - h\a A)^{-1} \1 U + h\a A(\id - h \a A)^{-1} B(W) \,.
  \label{e.aw}
\end{equation}

The first term of \eqref{e.aw} is of class \eqref{e.unifReg-Y1} by
Lemma~\ref{l.RK.Lambda}.  For the second term, we note that, by (B1)
and \eqref{e.RKShA.B},
\[
  \Sigma(W,U,h) = h \a A (\id - h \a A)^{-1} \, B(W)
\]
is of class
\[
  \Sigma \in \bigcap_{\substack{i+j+k \leq N \\ \ell \leq k \leq K}}
  \kCb^{(\underline i, \underline j, \ell)}
    ((\kD_k)^s \times \kY_{K+1} \times \kI; \kY_{k-\ell}^s) \,,
\]
so that $(\Sigma\circ W)^i$ is of class \eqref{e.unifReg-Y1} for
$i=1,\ldots, s$ by Lemma~\ref{l.cm-composition}. This proves
\eqref{e.unifReg-Y1}.

To prove that, for $\a$ invertible, $A\Psi$ is also of class
\eqref{e.unifReg-Y1}, we proceed analogously.  Applying $A$ to
\eqref{eq:NewRKUpdate}, we obtain
\begin{equation}
  \label{e.APsi}
  A \Psi^h(U) 
  = \mS(hA) AU  + h \b^T A (\id - h \a A)^{-1} B(W(U,h)) \,.
\end{equation}
The first term on the right of \eqref{e.APsi} is of class
\eqref{e.unifReg-Y1} by Lemma~\ref{l.RK.Lambda}.  We already proved
above that $V=\Sigma\circ W$ is of class \eqref{e.unifReg-Y1}.  As
$\a$ is invertible, $\b^T \a^{-1} V$ and, hence, \eqref{e.APsi} are of
class \eqref{e.unifReg-Y1}.

Next, we show improved regularity of $W$ and $\Psi$ with respect to
the step size in the same way as in the proof of
Theorem~\ref{t.local-diff-unif}.  Namely, we prove that
\[
  \partial_h W^i,\partial_h \Psi 
  \in \bigcap_{\substack{j+k \leq N-1 \\ \ell \leq k \leq K}}
  \kCb^{(\underline j, \ell)} (\kD_{K+1} \times [0,h_*];
  \kY_{k-\ell}) \,.
\]
Consider the $(K+1)$-scale of Banach spaces $\kZ_j =\kY^s_j$ for $j=0,
\dots, K$ and $\kZ_{K+1} = \kY^s_{K}$ with $\kW_{j} = \kD_j^s$ for
$j=0, \dots, K$ and $\kW_{K+1} = \kD^s_{K}$.  Set
$\kU=\kD_{K+1}^{-\delta}$, $\kX = \kY_{K+1}$, and $\kI=(0,h_*)$.  Due
to \eqref{e.RKShA.B2} and (B1), the map $\Pi$ from
\eqref{eq:newRKStagesIter} satisfies the assumptions of
Theorem~\ref{t.cm-scale} in this setting.  This shows that $\partial_h
W$ is of the above class, and proves, with \eqref{e.unifReg},
\eqref{e.unifReg-Y1} and Lemma~\ref{l.technical2} claim
\eqref{e.h-diff-unif-a} for the stage vector $W$.  Then
\eqref{e.RKShA.B2}, Lemma~\ref{l.RK.ShA}, \eqref{e.h-diff-unif-a},
Lemma~\ref{l.cm-composition} and Lemma \ref{l.DwPiClass} applied to
\begin{align*}
  \partial_h \Psi^h(U) 
  & = \partial_h \mS(hA) U
      + \b^T (\id-h\a A)^{-2} \, B(W(U,h))
      \notag \\
  & \quad
      + h\b^T(\id-h\a A)^{-1} \, \D B(W(U,h)) \, \partial_h W(U,h) \,,
\end{align*}
imply that $\partial_h \Psi$ is of the same class as $\partial_h W^i$.
When $\a$ is invertible, then, using that $A\Psi$ is of class
\eqref{e.unifReg-Y1} and using Lemma~\ref{l.technical2} as before,
claim \eqref{e.h-diff-unif-b} follows.

Statements \eqref{e.h-diff-unif-A} and \eqref{e.h-diff-unif-B} are, as
before, a consequence of Lemma~\ref{l.cm-scale}.
\end{proof}
  
\begin{remark}\label{r.h-diff-unif}
We actually showed in Theorem \ref{t.h-diff-unif} that for $\a$
invertible
\[
  W^i, \Psi \in \bigcap_{\substack{j+k \leq N \\ \ell \leq k \leq K}}
  \kCb^{(\underline j, \ell)} (\kD_{K+1}^{-\delta} \times [0,h_*]; 
  \kY_{k-\ell+1}),
\]
$i=1,\ldots, s$, i.e. $W^i$ and $\Psi$ have slightly higher regularity
in $U^0$ than the semiflow $\Phi^t$.
\end{remark}

\begin{remark}[Image of the numerical method and stage vector] 
Analogous to the situation for the semiflow noted in
Remark~\ref{r.local-diff-unif}, the proof of
Theorem~\ref{t.h-diff-unif} actually shows that, \mybf{when $\a$ is
invertible,}
\[
  \Psi, W^j \in \bigcap_{\substack{j+k \leq N \\ \ell \leq k \leq K +1\\ 
                              (k,\ell) \neq (K+1,0)}}
  \kCb^{(\underline j, \ell)} (\kD_{K+1}^{-\delta} \times [0,T_*]; 
  \mybf{  \kD_{k-\ell})}
\]
for $j=1,\ldots, s$.
\end{remark}

\begin{remark}\label{r.discr-improved-t-diff}
The proof of Theorem \ref{t.h-diff-unif} shows that, when  (A0), (A1), (B1) and
(RK2) hold,
but  $\a$ is not assumed invertible,
 we still have
 \[
 \Psi \in \bigcap_{\substack{j+k \leq N-1 \\   k \leq K }}
  \kCb^{(\underline j, k+1)} (\kD_{K+1}^{-\delta} \times [0,T_*]; 
    \kD). 
\] 
\end{remark}
 
\begin{remark} \label{r.flow} 
If $A$ generates a group rather than a semigroup, we may assume (A1)
for $h \in [-h_*,h_*]$ for some $h_*>0$.  Then Theorems~\ref{t.h-diff}
and~\ref{t.h-diff-unif} hold with $h \in [-h_*,h_*]$ (for some,
possibly, smaller choice of $h_*>0$).  Moreover, we can then also
weaken the requirement in (RK1) to $\lvert \mS(z) \rvert \leq 1$ for
$z \in \i \R$ and still show that (A2) implies (A1) for $h\in
[-h_*,h_*]$.  In this setting, the proof of Lemma~\ref{l.RK.Lambda2}
proceeds by recalling that, according to Remark~\ref{rem:HamPazy},
there exists $\omega>0$ such that $\lvert \Re (\spec A) \rvert \leq
\omega$.  Hence, we can decompose $\Q A$ into a skew-symmetric
operator $A_1 = \Im (A_\rn)$ and a bounded operator $A_2 = A_\rb+
\Re(A_\rn)$.
\end{remark}
 

\subsection{Convergence analysis of {\em A}-stable Runge-Kutta
methods}
\label{ss.convergence}

In this section we present a convergence analysis of $A$-stable
Runge--Kutta methods applied to semilinear evolution equation
\eqref{e.pde}.  The main difficulty is to prove differentiability in
the step size $h$ of the implicitly defined Runge--Kutta methods as
maps from a space of functions with higher regularity to a space with
lower regularity.
  
\begin{theorem}[Convergence] \label{c.RKConv-unif} 
Apply a Runge--Kutta method of classical order $p$ subject to
conditions \textup{(RK2)} and \textup{(A1)} to the semilinear
evolution equation \eqref{e.pde}.  Assume further that \textup{(B1)}
holds with $K \geq p$. Pick $\delta \in (0,\delta_*]$ such that $\kD_{p+1}^{-\delta}$
is non-empty and $T>0$.  Then there exist positive constants $h_*$,
$c_1$, and $c_2$ which only depend on the bounds afforded by
\textup{(B1)} \mybf{and \textup{(A1)}}, \eqref{e.RKShA.1}, on the
coefficients of the method, and on $\delta$, such that for every $U^0$
with
\begin{equation}
  \label{e.unifConvAss}
  \{ \Phi^t(U^0) \colon t \in [0,T] \} \subset \kD^{-\delta}_{p+1}
\end{equation}
and for every $h \in [0,h_*]$, the numerical solution
$(\Psi^h)^m(U^0)$ lies in $\kD$ and satisfies
\[
  \norm{(\Psi^h)^m(U^0) - \Phi^{mh}(U^0)}{\kY} 
  \leq c_2 \, \e^{c_1 mh} \, h^p
\]
so long as $mh \leq T$.  
\end{theorem}

\begin{proof}
We invoke Theorem~\ref{t.genConv} with $\kZ = \kY_{p+1}$, $\kD_\kZ =
\kD^{-\delta}_{p+1}$, $\kX = \kY$,
\[
  \kD_\kX 
  = \bigcup_{U \in \kD^{-\delta}_{p+1} } \kB_{R}^\kY(U)
  \subset \kD
\]
where $R=\delta$, and note that $\dist_\kX(\{U(t) \colon t \in
[0,T] \}, \partial \kD_\kX) \geq\delta$.  To verify the assumptions of
the theorem, we first note that local existence and regularity of a
solution to the evolution equation \eqref{e.pde} in the appropriate
spaces is always guaranteed by Theorem~\ref{t.local-diff-unif}.  In
particular, for initial data $U^0$ such that \eqref{e.unifConvAss}
holds, we also have $U \in \kC ([0,T]; \kD^{-\delta}_{p+1}) \cap
\kC^{p+1}([0,T]; \kD^{-\delta})$ with uniform bounds in the norms of
both spaces.  Conditions (C1) and (C2) follow from
Theorem~\ref{t.h-diff-unif} and Remark~\ref{r.discr-improved-t-diff}.
\end{proof}
 
\begin{remark} 
As explained in Sections~\ref{ss.swe} and~\ref{ss.nse}, the semilinear
wave equation and the nonlinear Schr\"odinger equation satisfy the
assumptions of Theorems~\ref{t.h-diff}--\ref{c.RKConv-unif} provided
the nonlinearity is sufficiently smooth.
\end{remark}

In the following corollary we prove the convergence of the $U$-derivatives
of the numerical solution. 
  
\begin{corollary}[Convergence of derivatives]
\label{c.conv-deriv}
In the setting of Theorem~\ref{c.RKConv-unif} there exist positive
constants $h_*$, $c_1$, and $c_2$ which only depend on the bounds
afforded by \textup{(B1)} \mybf{and \textup{(A1)}}, \eqref{e.RKShA.1},
on the coefficients of the method, and on $\delta$, such that for
every $U^0$ satisfying \eqref{e.unifConvAss} and for every $h \in
[0,h_*]$
\[
  \norm{\D^j_{U} (\Psi^h)^m (U^0)- \D^j_{U}\Phi^{mh}(U^0)}%
       {\kE^j(\kY_{p+1},\kY)} 
  \leq c_2 \, \e^{c_1 mh} \, h^p
\]
for $j \leq N-p-1$ so long as $mh \leq T$.  
\end{corollary}
 
\begin{proof}
We proceed by induction over $j$.  The case $j=0$ is already asserted
by Theorem~\ref{c.RKConv-unif}.  When $j>0$, we note that
$\tilde{U}(t) \equiv (U(t), W(t)) \equiv (\Phi^t(U^0), \D\Phi^t(U^0)
W^0)$ satisfies
\begin{equation}
  \label{e.var-pde}
  \frac{\d}{\d t}{\tilde{U}}(t) 
  = \tilde{A} \tilde{U} + \tilde{B}(\tilde{U})
\end{equation}
where
\[
  \tilde{A} = 
  \begin{pmatrix}
    A & 0 \\ 0 & A
  \end{pmatrix} \,, \qquad
  \tilde{B}(\tilde{U}) = 
  \begin{pmatrix}
    B(U) \\ \D B(U) W
  \end{pmatrix} \,,
\]
and we take $W^0 \in \kB_k \equiv \interior(\kB_1^{\kY_k}(0))$.
Similarly, the Runge--Kutta method applied to \eqref{e.var-pde}
satisfies
\[
  \tilde\Psi^h (\tilde{U}^0) = 
  \begin{pmatrix}
    \Psi^h(U^0) \\ \D \Psi^h(U^0) W^0
  \end{pmatrix}
  \quad \text{where} \quad
  \tilde{U}^0 = 
  \begin{pmatrix}
    U^0 \\ W^0
  \end{pmatrix} \,.
\] 
Eq.~\eqref{e.var-pde} and the Runge--Kutta method applied to it again satisfy
(A1), (B1) with $N$ replaced by $N-1$ and $\kD_k$ replaced by
$\kD_k \times \kB_k$ for $k=0, \dots, K$ and (RK2).  We can therefore apply the
induction hypothesis to the extended system so long as $j+p\leq N-1$.
\end{proof}


\appendix
 \section{Contraction mappings on a scale of Banach spaces}
\label{s.appendix}

In the appendix we present a contraction mapping theorem on a scale of
Banach spaces, our main technical tool.  \mybf{Our results are more
general than precursor versions in \cite{Vand87,Wulff00}.}  The proofs
are \mybf{technically involved} for two reasons.  First, there is some
combinatorial complexity in the estimates due to the implicitness of
the fixed point of the contraction map. For this reason we decided to
derive estimates in all required norms at once.  Second, the maps we
consider have derivatives with respect to the parameters that are only
strongly continuous, but not continuous in the operator norm.  This
precludes a straightforward induction argument.  What we find is that
this weaker notion of continuity is entirely sufficient, but requires
some extra care and notational effort.

For $K \in \N_0$, let ${\kZ} = {\kZ}_0 \supset {\kZ}_1 \supset\ldots
\supset {\kZ}_{K}$ be a scale of Banach spaces, each continuously
embedded in its predecessor, and let $\kV_j, \kW_j \subset \kZ_j$ be
nested sequences of sets.  Let $\kX$ be a Banach space, and let $\kU
\subset \kX$ and $\kI \subset \R$ be open.  We note that all results
in this section easily extend to the case where $\kI$ is an open
subset of $\R^p$.  Without loss of generality, we may assume that
$\lVert w \rVert_{\kZ_j} \leq \lVert w \rVert_{\kZ_{j+1}}$ for all $w
\in \kZ_{j+1}$.  (If this is not the case, we inductively equip
$\kZ_{j+1}$ with the equivalent norm $\lVert \, \cdot \,
\rVert_{\kZ_{j+1}} + \lVert \, \cdot \, \rVert_{\kZ_{j}}$.)  

We use the following additional integer indices.  The minimal
regularity we guarantee for the image space of the function considered
is the regularity of the \emph{lowest scale index} $L$ of the image,
the \emph{loss index} $S$ indicates how many rungs on the scale the
range of a function is down relative to its domain, and $N$ denotes
the maximal regularity of the function.  We assume $0 \leq L \leq K-S
\leq N-S$.  Taking the dependence on parameters into account, we work
with the family of spaces
\[
  \kC_{N,K,L,S}(\{\kV_j\}, \kU,\kI; \{ \kW_j\}) =
  \bigcap_{\substack{i+j+k \leq N-S \\ 
                     L + \ell \leq k \leq K-S}}
  \kCb^{(\underline i, \underline j, \ell)}
    ({\kV}_{k+S} \times{ \kU} \times  \kI; \kW_{k-\ell}) \,,
\] 
endowed with norm 
\begin{gather*}
  \Norm{\Pi}{N,K,L,S} 
  = \max_{\substack{i+j+k \leq N-S \\ 
          L+\ell \leq k \leq K-S}}
    \norm{\D_w^i \D_u^j \partial_\mu^\ell \Pi}{\kL_\infty
          (\kV_{k+S} \times \kU \times \kI;
           \kE^i(\kZ_{k+S},\kE^j(\kX; \kZ_{k-\ell})))}
\end{gather*}
for $0 \leq L \leq K-S \leq N-S$, and abbreviate
\begin{gather*}
  \kC_{N,K,L}(\{\kV_j\}, \kU,\kI; \{ \kW_j\}) 
  = \kC_{N,K,L,0}(\{\kV_j\}, \kU,\kI; \{ \kW_j\}) \,, \\
  \kC_{N,K}(\{\kV_j\}, \kU,\kI; \{ \kW_j \}) 
  = \kC_{N,K,0,0}(\{\kV_j\}, \kU,\kI; \{ \kW_j\})
\end{gather*} 
with corresponding norms 
\begin{gather*}
  \Norm{\Pi}{N,K,L} = \Norm{\Pi}{N,K,L,0} \,, \\
  \Norm{\Pi}{N,K} = \Norm{\Pi}{N,K,0,0} \,.
\end{gather*}
Note that any function of class $\kC_{N,K,L,S}$ has a maximal number
of $N-L-S$ derivatives in its first and second argument on the lowest
admissible domain scale $\kZ_{L+S}$.

Furthermore, let
\[
  \kC_{N,K,L}( \kU,\kI; \{ \kW_j \})
  = \bigcap_{\substack{j+k \leq N \\ L+\ell \leq k \leq K}}
      \kCb^{(\underline j,\ell)} (\kU \times \kI; \kW_{k-\ell}) \,,
\]
endowed with norm 
\begin{gather}
  \label{e.norm3w}
  \norm{w}{N,K,L} 
  = \max_{\substack{j+k \leq N \\ 
          L+\ell \leq k \leq K}}
    \norm{\D_u^j \partial_\mu^\ell w}{\kL_\infty
          (\kU \times \kI; \kE^j(\kX; \kZ_{k-\ell}))} 
\end{gather}
for $0 \leq L \leq K \leq N$, where we abbreviate   
\[
  \kC_{N,K}( \kU,\kI; \{\kW_j\}) = \kC_{N,K,0}( \kU,\kI; \{\kW_j\})
\]
with corresponding norm 
\begin{gather*}
  \norm{w}{N,K} = \norm{w}{N,K,0} \,.
\end{gather*}
For future reference, we note the following.

\begin{remark} \label{r.no-w} 
When a map $\Pi \in \kC_{N,K,L,S}(\{\kV_j\}, \kU,\kI; \{ \kW_j\})$
does not depend on $w$, it can be interpreted as an element from
$\kC_{N,K,L}(\kU,\kI; \{ \kW_j\})$ where
\[
  \Norm{\Pi}{N,K,L,S} = \norm{\Pi}{N-S,K-S,L} \,.
\]
\end{remark}

We simply write $\kC_{N,K,L,S}$ and $\kC_{N,K,L}$ when the arguments
are unambiguous.  We also write
\[
  \partial_\mu \Pi(w(u,\mu);u,\mu) 
  = \partial_\mu \Pi(w;u,\mu) \big|_{w=w(u,\mu)}
  = (\partial_\mu \Pi \circ w)(u,\mu)
\] 
to denote partial $\mu$-derivatives vs.\ $\D_\mu
(\Pi(w(u,\mu),u,\mu))$ to denote full $\mu$-derivatives.

We begin with four short technical lemmas.  The first specifies the
relation between the spaces $\kC_{N,K}$ and $\kC^K$.
 
\begin{lemma} \label{l.cm-scale}
If $N>K$  then, with $\kW \equiv \kW_0$,
\[
  \kC_{N,K}(\kU,\kI;\{\kW_j\}) \subset \kCb^K (\kU \times \kI; \kW) \,.
\]
\end{lemma}

\begin{proof} 
Let $w \in \kC_{N,K}(\kU,\kI;\{\kW_j\})$.  Fixing $\ell=k$ in the
definition of $\kC_{N,K}$ and recalling \eqref{e.cont-b}, i.e., strong
and uniform continuity coincide if no derivative in $u$ is taken, we
find that
\[
  w \in \kCb^{(0,K)}(\kU \times \kI;   \kW) 
  \cap \bigcap_{\substack{j+\ell \leq K+1 \\ \ell \leq K}}
  \kCb^{(\underline{j},\ell)}(\kU \times \kI;   \kW) \,.
\]
The claimed uniform continuity then holds because of \eqref{e.cont-a}.
\end{proof}

The following lemma captures   the essence of the inductive
step in $N$ as needed in the main results which follow.

\begin{lemma} \label{l.technical2} 
If $w \in \kC_{N,K,L}(\kU,\kI;\{\kW_j\})$ and the map
$(u,\tilde{u},\mu) \mapsto \D_u w(u,\mu) \tilde{u}$ is of class
$\kC_{N,K,L}(\kU \times \kB^\kX_1(0),\kI; \{\kZ_j\})$, then $w \in
\kC_{N+1,K,L}(\kU,\kI;\{\kW_j\})$ and
\[
  \norm{w}{N+1,K,L} 
  \leq \sup_{\|\tilde{u}\|_{\kX} \leq 1}
  \norm{\D_u w \, \tilde{u}}{N,K,L} + \norm{w}{N,K,L} \,.
\]
\end{lemma}

\begin{proof}
The claim is a direct consequence of the partitioning of the index set
in the definition of the $(N+1,K,L)$-norm, see \eqref{e.norm3w}, into
\[
  \{ 0 \leq j + k \leq N+1 \}
  = \{ 0 \leq j + k \leq N \} \cup \{ 0 \leq \tilde j + k \leq N \}
\]
where $\tilde j = j-1$, using the definition of the operator norm,
\[
  \norm{T}{\kE(\kX,\kY)} 
  = \sup_{\lVert x \rVert_\kX=1} \norm{Tx}{\kY} \,,
\]
and the definition of the $(N,K,L)$-norm \eqref{e.norm3w}.
\end{proof}

The next lemma captures the essence of the inductive step in $K$.
Namely, a scale of length $K+1$ can be broken up into two scales which
have only length $K$, plus a trivial remaining bit.
 
\begin{lemma} \label{l.technical1}
When $N>K$, $w \in \kC_{N,K+1,L+1}(\kU,\kI;\{\kW_j\}) \cap
\kC_{N,L,L}(\kU,\kI;\{\kW_j\})$, and $\partial_\mu w \in
\kC_{N-1,K,L}(\kU,\kI;\{\kZ_j\})$,
 then $w \in \kC_{N,K+1,L}(\kU,\kI;\{\kW_j\})$ and
\[
  \norm{w}{N,K+1,L} 
  \leq \norm{w}{N,K+1,L+1} 
       + \norm{w}{N,L,L}
       + \norm{\partial_\mu w}{N-1,K,L} \,.
\]
\end{lemma}

\begin{proof}
Translating the scale, i.e., setting $\tilde\kZ_j = \kZ_{j+L}$,
$\tilde{K}=K-L$, and $\tilde N = N-L$, we can reduce to the case $L=0$.
Since
\[
  \{ 0\leq \ell \leq k \leq K+1 \} 
  = \{ 0\leq  \ell < k \leq K+1 \} 
    \cup \{ 1\leq \ell \leq k \leq K+1 \} 
    \cup \{k=\ell=0 \}
\]
and $\partial_\mu w \in \kC_{N-1,K}$ if and only if
\[
  w \in \bigcap_{\substack{j+k \leq N-1 \\ \ell \leq k \leq K}}
    \kCb^{(\underline j, \ell+1)} (\kU \times \kI; \kW_{k-\ell})
  = \bigcap_{\substack{j+k \leq N \\ 1 \leq \ell \leq k \leq K+1}}
    \kCb^{(\underline j, \ell)} (\kU \times \kI; \kW_{k-\ell}) \,,
\]
the claim follows directly from definition of $\kC_{N,K,L}$ and its
norm \eqref{e.norm3w}.
\end{proof}
 
Finally, we prove that the space $\kC_{N,K,0,S}$ can be expressed in
terms of $\kC_{N,K}$-type spaces with domains defined on a scale.

\begin{lemma} \label{l.technical3}
We have
\[
  \bigcap_{S \leq \kappa\leq K} \kC_{N-S,\kappa-S,L}(\kV_\kappa \times
    \kU, \kI; \{ \kW_j\})
  = \kC_{N,K,L,S}(\{\kV_j\},\kU,\kI; \{\kW_j\}) \,,
\]
and
\[
  \Norm{\Pi}{\kC_{N,K,L,S}(\{\kV_j\},\kU;\kI;\{ \kW_j \})} 
  \sim \max_{S \leq \kappa \leq K}
       \norm{\Pi}{\kC_{N-S,\kappa-S,L}(\kV_\kappa \times \kU,\kI;
                  \{ \kW_j \})} \,,
\]
where $\sim$ denotes that left hand and right hand sides provide
equivalent norms on $\kC_{N,K,L,S}$.
\end{lemma}

\begin{proof}
Translating the scale, i.e., setting $\tilde\kZ_j = \kZ_{j+L}$,
$\tilde{K}=K-L$, and $\tilde N = N-L$, we can reduce to the case
$L=0$.  Next, we identify
\begin{align*}
  \bigcap_{S \leq \kappa\leq K} 
    \kC_{N-S,\kappa-S}(\kV_\kappa \times \kU;\kI, \{ \kW_j\})
  & = \bigcap_{\substack{S \leq \kappa\leq K\\j+k \leq N-S \\ 
               \ell \leq k \leq \kappa-S}}
      \kCb^{(\underline j, \ell)} 
        ((\kV_\kappa\times \kU)\times \kI; \kW_{k-\ell}) \\
  & = \bigcap_{\substack{S \leq \kappa\leq K\\i+j+k \leq N-S \\ 
               \ell \leq k \leq \kappa-S}}
      \kCb^{(\underline i,\underline j, \ell)} 
        (\kV_\kappa\times \kU\times \kI; \kW_{k-\ell}) \\
  & = \bigcap_{\substack{0 \leq \tilde{k}\leq K-S\\i+j+k \leq N-S \\ 
               \ell \leq k \leq \tilde{k}}}
      \kCb^{(\underline i,\underline j, \ell)} 
        (\kV_{\tilde{k}+S} \times \kU\times \kI; \kW_{k-\ell}) \\
  & = \bigcap_{\substack{i+j+k \leq N-S \\ 
               \ell \leq k \leq K-S}}
      \kCb^{(\underline i, \underline j, \ell)}
        ({\kV}_{k+S} \times{ \kU} \times  \kI; \kW_{k-\ell}) \,,
\end{align*}
which equals $\kC_{N,K,0,S}$.  Noting that
\begin{multline*}
  \max_{\substack{S \leq \kappa \leq K \\ j+k \leq N-S \\
            \ell \leq k \leq \kappa - S}}
      \norm{\D_{(w,u)}^j \partial_\mu^\ell \Pi}{\kL_\infty
            (\kV_\kappa \times \kU \times \kI; 
             \kE^j(\kZ_\kappa \times \kX; \kZ_{k-\ell}))} \\
  \sim \max_{\substack{S \leq \kappa \leq K \\ i+j+k \leq N-S \\
            \ell \leq k \leq \kappa - S}}
      \norm{\D_w^i \D_u^j \partial_\mu^\ell \Pi}{\kL_\infty
            (\kV_\kappa \times \kU \times \kI; 
             \kE^i(\kZ_\kappa; \kE^j(\kX; \kZ_{k-\ell})))} \,,
\end{multline*}
the statement about the norms follows analogously.
\end{proof}

The next lemma will be our main tool for obtaining estimates on the
scale of Banach spaces for compositions of maps of the form
\[
  (\Pi\circ w)(u,\mu) \equiv \Pi(w(u,\mu);u,\mu) \,.
\]
The essence of the result is very natural: When the outer function
$\Pi$ loses $S$ rungs on the scale, the inner function $w$ must have
minimal regularity $L=S$ and the composition maps at best into  
rung $K-S$.

The main difficulty in the proof of this lemma and of the subsequent
results is that the maps considered lose smoothness when derivatives
in $\mu$ are taken.  In particular, these derivatives are only
strongly continuous with respect to the parameters $u$ and $\mu$ and,
in our infinite-dimensional setting, are discontinuous with respect to
$u$ and $\mu$ in the operator norm.  As a result, in the proofs below
the induction hypothesis cannot be applied to the derivatives in a
straightforward way.

\begin{lemma}[Chain rule on a scale of Banach spaces]
\label{l.cm-composition} 
Let $\Pi = \Pi(w;u,\mu)$ and $w=w(u,\mu)$ satisfy
\[
  \Pi \in \kC_{N,K,L,S}(\{\kW_j\}, \kU, \kI; \{\kZ_j\})
  \quad \text{and} \quad
  w \in \kC_{N,K,S+L} (\kU, \kI; \{\kW_j\}) \,.
\]
Then $\Pi \circ w \in \kC_{N-S,K-S,L} (\kU,\kI; \{\kZ_j\})$ and
$\lVert \Pi\circ w \rVert_{N-S,K-S,L}$ can be bounded by a polynomial
with non-negative coefficients in $\Norm{\Pi}{N,K,L,S}$ and
$\lVert w \rVert_{N,K,S+L}$. 
\end{lemma}
 
\begin{proof}
Translating the scale, i.e., setting $\tilde\kZ_j = \kZ_{j+L}$,
$\tilde{K}=K-L$, and $\tilde N = N-L$, we can reduce to the case $L=0$
as in the proof of Lemma~\ref{l.technical1}.  We proceed by induction
in $N$ and $K$ as follows.

For $N=K=S$, we have $\Pi \in \mybf{\kCb}(\kW_S\times \kU\times\kI; \kZ_0)$,
$w \in \mybf{\kCb}(\kU\times\kI;\kW_S)$, hence $\Pi \circ w \in \mybf{\kCb} (\kU,\kI;
\kZ_0)$ with bound $\Norm{\Pi}{S,S,0,S}$.
 
Let us now increment $N$ holding $K$ and $S$ fixed.  Let $\kB \equiv
\kB_1^{\kX}(0)$.  We claim that the map
\begin{equation}
  \label{e.DuPi}
  (u, \tilde u, \mu) \mapsto \D_u (\Pi \circ w) \, \tilde u
  \quad \text{is of class} \quad
  \kC_{N-S,K-S}(\kU \times \kB, \kI; \{\kZ_j\}) 
\end{equation}
with a bound which is a polynomial in $\Norm{\Pi}{N+1,K,0,S}$ and
$\lVert w \rVert_{N+1,K,S}$.  The inductive step is achieved by
Lemma~\ref{l.technical2} which then asserts that $\Pi \circ w \in
\kC_{N+1-S,K-S}$ with its norm bounded as required.

To prove this claim, let $\tilde{u} \in \kB$, write
\begin{equation}
  \D_u (\Pi \circ w)(u,\mu) \, \tilde{u}
  = \partial_u \Pi (w(u,\mu); u, \mu) \, \tilde{u}
    + \partial_w \Pi (w(u,\mu); u, \mu) \,
      \D_u w(u,\mu) \, \tilde{u} \,,
  \label{e.du-chainrule}
\end{equation}
and consider each term on the right of \eqref{e.du-chainrule}
separately.  For the first term on the right, set $\hat{u} \equiv
(u,\tilde{u}) \in \kU \times \kB \equiv \hat\kU$ and define
\begin{equation}
  \label{e.Pi1a}
  \Pi_1 (w; \hat{u}, \mu) 
  = \partial_u \Pi (w; u, \mu) \, \tilde{u} \,.
\end{equation}
By assumption, this map is of class $\kC_{N,K,0,S} (\{\kW_j\},
\kU\times\kB, \kI; \{\kZ_j\})$ .  The induction hypothesis, applied to
the maps $\Pi_1$ and $w$, then asserts that
\begin{equation}
  \label{e.Pi1b}
  \Pi_1 \circ w \in \kC_{N-S,K-S}(\kU\times\kB, \kI; \{\kZ_j\})
\end{equation}
and that its $\kC_{N-S, K-S}$-norm is bounded by a polynomial with
non-negative coefficients in $\Norm{\Pi}{N+1,K,0,S} \geq
\Norm{\Pi_1}{N,K,0,S}$ and $\lVert w \rVert_{N,K,S}$.
 
For the second term on the right of \eqref{e.du-chainrule}, we must
proceed in stages.  Fix $r = \lVert w \rVert_{N+1,K,S}$ and let
$\kV_\kappa = \kB_r^{\kZ_\kappa}(0)$.  For $\kappa = S, \dots, K$ and
$(u,\hat{w}) \in \kU \times \kV_\kappa$, we set
\[
  \Pi_2(w; (u, \hat{w}), \mu) 
  = \D_w \Pi (w,u,\mu) \, \hat{w} \,.
\]
By assumption, this map is of class $\kC_{N,\kappa,0,S}(\{\kW_j\}, \kU
\times \kV_\kappa, \kI; \{\kZ_j\})$.  The induction hypothesis,
applied to the maps $\Pi_2$ and $w$, then asserts that
\begin{equation}
  \Pi_2 \circ w 
  \in \kC_{N-S, \kappa-S} (\kU \times \kV_\kappa, \kI;  \{\kZ_j\}) 
  \label{e.pi2class}
\end{equation}
and that its $\kC_{N-S, \kappa-S}$-norm is bounded by a polynomial in
$\lVert w \rVert_{N,K,S} \geq \lVert w \rVert_{N,\kappa,S}$ and
\[
  \Norm{\Pi}{N+1,K,0,S} \, 
  \sup_{\hat{w} \in \kV_\kappa} \norm{\hat{w}}{\kZ_\kappa} 
  \geq \Norm{\Pi_2}{\kC_{N,\kappa,0,S} 
         (\{ \kW_j \}, \kU \times \kV_\kappa,\kI; \{ \kZ_j \})} \,.
\]
We now consider the composition $\Pi_2 \circ w$ as a map 
\[
  \hat \Pi (\hat w; u, \mu) 
  = \partial_w \Pi (w(u,\mu);u,\mu) \, \hat w \,.
\]
Recalling that \eqref{e.pi2class} applies for all $\kappa = S, \dots,
K$, we can apply Lemma~\ref{l.technical3} to obtain that
\[
  \hat \Pi \in \kC_{N,K,0,S} (\{\kV_j\}, \kU,\kI;\{\kZ_j\})
\]
and that its norm is bounded by a polynomial in $r \,
\Norm{\Pi}{N+1,K,0,S}$ and $\tnorm{w}{N,K,S}$.  (This is summarized in
Lemma~\ref{l.DwPiClass} for later use.)

Now consider $\hat \Pi$ as a function of $\hat{w}$, $\hat{u} =
(u,\tilde{u})\in \kU\times\kB$, and $\mu$.  Since
\[
  \norm{\D_uw(u,\mu) \, \tilde u}{\kCb(\kU \times \kI; \kZ_j)}
  \leq \norm{w}{N+1,K,S} \, 
       \norm{\tilde u}\kX \leq r 
\]
for $j=S,\dots,K$, the function $\hat w (\hat u, \mu) = \D_u w(u,\mu)
\, \tilde u$ is of class $\kC_{N,K,S} (\kU\times \kB, \kI;
\{\kV_j\})$.  Applying the induction hypothesis to $\hat \Pi$ and
$\hat w$, we conclude that $\hat \Pi \circ \hat w$ or, written
explicitly, the map
\[
  ((u,\tilde{u}),\mu) \mapsto 
  \partial_w \Pi (w(u,\mu); u, \mu) \, \D_u w (u, \mu) \, \tilde{u}
\]
is of class $\kC_{N-S,K-S}(\kU\times \kB,\kI; \{\kZ_j\})$, with norm
bounded by an increasing polynomial in $\Norm{\Pi}{N+1,K,0,S}$ and
$\tnorm{w}{N+1,K,S} \geq \tnorm{\hat{w}}{N,K,S}$.  Due to
\eqref{e.du-chainrule}, \eqref{e.Pi1a}, and \eqref{e.Pi1b}, this also
holds for the map $((u,\tilde{u}),\mu) \mapsto \D_u \Pi (w(u,\mu); u,
\mu) \, \tilde{u}$, thus proves our claim \eqref{e.DuPi}; the
inductive step in $N$ is complete.

Next, we increment $K-S$ keeping $N$ fixed.  Here the inductive step
will be achieved by Lemma~\ref{l.technical1}; we must hence verify its
assumptions.  First, applying the induction hypothesis on the scale
$\tilde \kZ_j = \kZ_{j+1}$ with $j=0,\ldots, K$, we infer that
\[ 
  \Pi \circ w \in \kC_{N-1-S,K-S}(\kU,\kI; \{\tilde\kZ_j\}) =  
  \kC_{N-S,K+1-S,1}(\kU,\kI; \{\kZ_j\})
\]
with the corresponding norm bounded by a polynomial with non-negative
coefficients in $\Norm{\Pi}{N,K+1,0,S} \geq \Norm{\Pi}{N,K+1,1,S+1}$
and $\tnorm{w}{N,K+1,S+1}$.  Second, by the induction hypothesis
applied on the trivial scale,
\[
  \Pi \circ w  \in \kC_{N-S,0}(\kU,\kI; \{\kZ_j\}) \,,
\]
with the corresponding norm bounded by a polynomial with non-negative
coefficients in $\Norm{\Pi}{N,K+1,0,S} \geq \Norm{\Pi}{N,0,0,S,0}$ and
$\tnorm{w}{N,K+1,S} \geq \tnorm{w}{N,0,S}$.  Third, we claim that
\begin{equation}
  \label{e.DmuPiw1.c}
  \D_\mu (\Pi\circ w)  \in \kC_{N-1-S,K-S}(\kU,\kI; \{\kZ_j\}) \,,
\end{equation}
with the corresponding norm bounded by a polynomial with non-negative
coefficients in $\Norm{\Pi}{N,K+1,0,S}$ and $\tnorm{w}{N,K+1,S}$.
Then Lemma~\ref{l.technical1} applied to $\Pi \circ w$ where $N$ and
$K$ there correspond to $N-S$ and $K-S$ here proves that $\Pi \circ
w\in \kC_{N-S,K+1-S}$ with the required bound on its norm; this
concludes the inductive step.

It remains to prove claim \eqref{e.DmuPiw1.c}.  Following the steps
in the estimate of $\D_u(\Pi \circ w)$ above, we write
\begin{equation}
  \label{e.DmuPiw}
  \D_\mu (\Pi \circ w)(u,\mu) 
  = \partial_\mu \Pi (w(u,\mu); u, \mu) 
    + \partial_w \Pi (w(u,\mu); u, \mu) \, \D_\mu w 
\end{equation}
and consider each term on the right of \eqref{e.DmuPiw} separately.
For the first term, note that the assumption on $\Pi$ implies, in
particular, that $\partial_\mu \Pi \in \kC_{N,K+1,0,S+1}$ and that, by
assumption, $w \in \kC_{N, K+1, S+1}$.  Since $K-S$ is not increased,
the induction hypothesis applies to this pair of maps and yields
\begin{equation}
  \label{e.DmuPiw1}
  \partial_\mu \Pi \circ w
  \in \kC_{N-S-1,K-S} (\kU,\kI; \{\kZ_j\})
\end{equation}
with a polynomial bound in $\Norm{\Pi}{N,K+1,0,S}\geq
\Norm{\partial_\mu\Pi}{N,K+1,0,S+1}$ and $\tnorm{w}{N,K+1,S} \geq
\tnorm{w}{N,K+1,S+1}$.

For the second term on the right of \eqref{e.DmuPiw}, fix $r = \lVert
w \rVert_{N,K+1,S}$ and let $\kV_j= \kB_r^{\kZ_j}(0)$ for $j = S,
\dots, K+1$.  We saw above that the map $\hat\Pi$ from \eqref{e.hatPi}
is of class $\kC_{N-1,K,0,S}(\{\kV_j\},\kU,\kI;\{\kZ_j\})$ with norm
bounded as specified in Lemma \ref{l.DwPiClass}.  The assumption on
$w$ and the definition of $r$ above imply, moreover, that
\[
  \D_\mu w \in \kC_{N-1,K,S} (\kU, \kI; \{ \kV_j \}) \,.
\]
Thus, the induction hypothesis applied once more to this pair of maps
yields
\begin{equation}
  \hat \Pi \circ \partial_\mu w = \D_w (\Pi \circ w)\D_\mu w 
  \in \kC_{N-1-S,K-S} (\kU,\kI; \{\kZ_j\}) \,
  \label{e.DmuPiw2}
\end{equation}
with a polynomial bound in $\Norm{\Pi}{N,K+1,0,S}$ and
$\tnorm{w}{N,K+1,S} \geq \tnorm{\partial_\mu w}{N,K,S}$.  Together,
\eqref{e.DmuPiw1} and \eqref{e.DmuPiw2} imply \eqref{e.DmuPiw1.c} with
the required bound.
\end{proof}

In the proof of Lemma~\ref{l.cm-composition}, we implicitly proved the
following result which we state here for later reference.

\begin{lemma} \label{l.DwPiClass}
Let $\Pi$ and $w$ satisfy the conditions of
Lemma~\ref{l.cm-composition} with $L=0$; let  $r>0$ and $\kV_j
=\kB_r^{\kZ_j}(0)$ for $j=0,\ldots, K$.  Then  
\begin{equation}
  \label{e.hatPi}
  \hat \Pi (\hat{w}; u, \mu) 
  \equiv \D_w \Pi(w(u,\mu); u, \mu) \, \hat{w}
\end{equation}
satisfies
\[
  \hat \Pi \in \kC_{N-1,K,0,S}(\{\kV_j\}, \kU,\kI; \{ \kZ_j\})
\]
with a polynomial bound in $\tnorm{w}{N-1,K,S}$ and
$r \, \Norm{\Pi}{N,K,0,S}$.
\end{lemma}

\begin{remark}
The Fa\`a di Bruno formula (see, e.g., \cite{FaadiBruno}) can be used
to compute the derivatives of compositions of functions explicitly.
However, it does not remove the need to estimate complete $\kC_{N,K}$
norms.  Thus, an inductive argument seems to be the most manageable
way of writing out a proof.
\end{remark}

We now proceed to the crucial contraction mapping theorem for maps
$\Pi(\,\cdot\,; u, \mu)$ of class $\kC_{N,K}$.
  
\begin{theorem}[Contraction mappings on a scale of Banach spaces]
\label{t.cm-scale}
For $N, K \in \N_0$ with $N\geq K$, let ${\kZ} = {\kZ}_0 \supset
{\kZ}_1 \supset\ldots \supset {\kZ}_{K}$ be a scale of Banach spaces,
each continuously embedded in its predecessor, let $\kW_j\subset\kZ_j$
be a nested sequence of closures of open sets, let $\kX$ be a Banach
space, and let $\kU \subset \kX$ and $\kI \subset \R$ be open.  Let
$(w, u,\mu) \mapsto \Pi(w; u,\mu)$ be a nonlinear map such that
\begin{itemize}
\item[(i)] $\Pi \in \kC_{N,K}(\{\kW_j\}, \kU,\kI; \{ \kW_j\})$;
\item[(ii)] $w \mapsto \Pi(w; u, \mu)$ is a contraction on ${\kW}_j$
with contraction constant $c_j'<1$ uniformly for all $u \in {\kU}$,
$\mu \in \kI$, and $j = 0, \dots, K$.
\end{itemize}
Then the fixed point equation $\Pi(w; u,\mu) = w$ has a unique solution
\[
  w \in \kC_{N,K} (\kU, \kI; \{  {\kW}_j \})
\]
and $\tnorm{w}{N,K}$ is bounded by a function which is a polynomial
with non-negative coefficients in $\Norm{\Pi}{N,K}$ and
$(1-c_j')^{-1}$.
\end{theorem}
 
Similar theorems were proved in \cite{Vand87} for the case $K=1$,
$\mybf{\kU}= \emptyset$ and in \cite{Wulff00} for the case $N=K\in\N$,
$\mybf{\kU}= \emptyset$.  Due to Lemma \ref{l.cm-scale}, the theorem
as stated here implies, in particular, that $w \in \kCb^K (\kU \times
\kI; \kW)$.  This simple statement on $\kC^K$ differentiability is
reminiscent of the standard form of the contraction mapping theorem
with parameters as, for example, stated in \cite[p.~13]{Henry}.

\begin{proof}[Proof of Theorem~\ref{t.cm-scale}]
The argument is once more an induction in $N$ and $K$, following the
combinatorial pattern of the proof of Lemma~\ref{l.cm-composition}.
For $N=K=0$, the regular contraction mapping theorem with parameters
asserts that $w \in \kC(\kU\times \kI; \kW)$.  Moreover, $\Pi$ is a
contraction uniformly for $(u,\mu) \in \cl(\kU)\times\cl(\kI)$ so that
$\Pi$ has a unique fixed point $w(u_*,\mu_*)$ also for $(u_*,\mu_*)$
on the boundary of $\kU \times \kI$.  From this, a straightforward
estimate yields continuity of $w$ up to the boundary; thus, $w \in
\kC_{0,0}(\kU,\kI;\{ \kW_0\})$.

Assume now that the conclusion of the theorem holds for fixed $K$ and
$N\geq K$.  We first employ Lemma~\ref{l.technical2} to show that the
conclusion also holds when we increment $N$, holding $K$ fixed.

As in the proof of Lemma~\ref{l.cm-composition}, we set $\kB \equiv
\kB_1^{\kX}(0)$ and let $(u,\tilde{u}) \in \kU \times \kB \equiv
\tilde\kU$.  Differentiating the fixed point equation $w=\Pi \circ w$
with respect to $u$, we find that $\D_u w \, \tilde u$ formally solves
the fixed point equation $\tilde w = \tilde \Pi (\tilde w;
(u,\tilde{u}), \mu)$, where
\begin{align*}
  \tilde \Pi (\tilde w; (u,\tilde{u}), \mu)
  & = \partial_w \Pi (w(u,\mu); u, \mu) \, \tilde w
      + \partial_u \Pi (w(u,\mu); u, \mu) \, \tilde u 
    \notag \\
  & \equiv \hat \Pi (\tilde w; u, \mu) 
      + \partial_u \Pi (w(u,\mu); u, \mu) \, \tilde u \,.
\end{align*}
Using the chain rule Lemma~\ref{l.cm-composition} and
Lemma~\ref{l.DwPiClass}, we infer that
\[
  \tilde \Pi \in \kC_{N,K}(\{\kV_{j}\},\kU\times \kB,\kI;\{\kZ_{j}\})
\]
with $\kV_j = \kB_r^{\kZ_j}(0)$ for $j = 0, \dots, K$ and arbitrary
$r>0$.  Here, we must prove in addition that $\tilde \Pi$ maps each of
the $\kV_0, \dots, \kV_K$ into itself.  Indeed, a direct estimate
shows that it suffices to take
\[
  r = \Norm{\Pi}{N+1,K,S} \, \max_{j=0,\ldots, K} \frac1{1-c'_j}
  \geq \max_{j=0,\ldots, K} 
       \frac{\norm{\partial_u \Pi \circ w}%
            {\kL_\infty(\kU\times \kI; \kE(\kX,\kZ_j))}}{1-c'_j} \,. 
\]
The induction hypothesis then applies to $\tilde \Pi \in \kC_{N,K} (\{
\kV_j \}, \kU\times \kB, \kI; \{ \kV_j \})$, yielding the existence of
a fixed point $\tilde w \in \kC_{N,K}(\kU\times \kB,\kI;\{\kV_{j}\})$.
  
It remains to be shown that the formal identity $\tilde w = \D_u w \,
\tilde u$ holds true on each $\kZ_j$ for $j=0,\ldots, K$.  This,
however, follows by \cite[Theorem 4.8]{Wulff00} (see also the proof of
\cite[Theorem 3]{Vand87}) applied to the one-parameter family of maps
$(w; \nu) \mapsto \Pi(w; u + \nu\tilde{u}, \mu)$ for fixed $\mu
\in\kI$, $u \in \kU$, and $\tilde{u} \in \kB$ on the scale
$\{\tilde\kZ_0, \tilde\kZ_1\} = \{\kZ_j, \kZ_j\}$ for each $j = 0,
\ldots, K$. 

Altogether, since $w \in \kC_{N,K}(\kU,\kI; \{\kW_j \})$,
Lemma~\ref{l.technical2} applies and yields gives $w \in
\kC_{N+1,K}(\kU,\kI; \{ \kW_j \})$; the inductive step in $N$ is
complete.

Next, we increment $K<N$ holding $N$ fixed.  For this, we use
Lemma~\ref{l.technical1}.  First, we note that assumptions (i) and
(ii) hold on the $K$-step scale $\kZ_1 \supset \dots \supset
\kZ_{K+1}$ so that the induction hypothesis applies; we find that
\[
  w \in \kC_{N,K+1,1} (\kU \times \kI; \{ \kW_{j}\} ) \,.
\]
Second, by the induction hypothesis applied on the trivial scale, $w
\in \kC_{N,0}$.  Third, differentiating the fixed point equation $w =
\Pi \circ w$ with respect to $\mu$, we obtain that $\partial_\mu w$
formally solves the fixed point equation $\tilde w = \tilde \Pi
(\tilde w; u, \mu)$, where
\begin{align*}
  \tilde \Pi (\tilde w; u, \mu)
  & = \partial_w \Pi (w(u,\mu); u, \mu) \, \tilde w
      + \partial_\mu \Pi (w(u,\mu); u, \mu) 
    \notag \\
  & \equiv \hat \Pi (\tilde w; u, \mu) 
      + \partial_\mu \Pi (w(u,\mu); u, \mu) \,.
\end{align*}
Since, by assumption, $w \in \kC_{N,K}$, we infer from
Lemma~\ref{l.cm-composition} and Lemma \ref{l.DwPiClass} that
\[
  \tilde \Pi 
  \in \kC_{N-1,K} (\{ \kV_j \},  \kU\times\kB, \kI; \{ \kZ_j \}) \,.
\]
Here, we need in addition that $\tilde \Pi$ maps each $\kV_0, \ldots,
\kV_K$ into itself.  This is satisfied whenever
\[ 
  r = \Norm{\Pi}{N,K+1,S} \, \max_{j=0,\ldots, K} \frac1{1-c'_j}
  \geq \max_{j=0,\ldots, K}
       \frac{\norm{\partial_\mu \Pi \circ w}%
            {\kL_\infty(\kU\times\kI;\kZ_j)}}{1-c'_j} \,.
\]
The induction hypothesis then applies to $\tilde \Pi \in \kC_{N-1,K}
(\{ \kV_j \}, \kU\times\kB, \kI; \{ \kV_j \})$, yielding the existence
of a fixed point $\tilde w \in \kC_{N-1,K}(\kU\times
\kB,\kI;\{\kV_{j}\})$.  By \cite[Theorem 4.8]{Wulff00} (see also the
proof of \cite[Theorem 3]{Vand87}), applied to $(w; \mu) \mapsto
\Pi(w;u,\mu)$ for each fixed $u \in \kU$ on the two-step scale
$\{\tilde\kZ_0, \tilde \kZ_1\} \equiv \{\kZ_j, \kZ_{j+1}\}$ for each
$j=0,\ldots, K$, we ensure that the formal identity $\tilde w
= \partial_\mu w$ holds true across the scale $\kZ_0, \dots \kZ_K$.
We conclude that $\partial_\mu w\in \kC_{N-1,K}$.

Altogether, Lemma~\ref{l.technical1} applies and yields $w \in
\kC_{N,K+1}$; the inductive step in $K$ is now complete.  We note that
the required polynomial bounds are obtained, as before, by carefully
tracking all the bounds in the respective norms through the argument.
We omit all detail.
\end{proof}

\section*{Acknowledgments}

C.W.~thanks the Courant Institute of Mathematical Sciences for their
hospitality during the preparation of parts of the manuscript, and
acknowledges funding by the Nuffield Foundation, by the Leverhulme
Foundation and by EPSRC grant EP/D063906/1.  M.O.~was visiting the
Courant Institute of Mathematical Sciences supported by a Max--Kade
Fellowship when part of this work was done, and further acknowledges
support through the ESF network Harmonic and Complex Analysis and
Applications (HCAA).


\bibliographystyle{plain}

\end{document}